\documentclass[12pt,a4paper]{article}
\usepackage{amssymb,times1,amsmath,amsfonts,amsthm,euscript}
\usepackage{hyperref}
\usepackage{graphicx}
\usepackage{mysec}
\usepackage{color}
\usepackage[normalem]{ulem}
\usepackage{cancel}

\makeatletter
\@addtoreset{equation}{section}

\makeatother

\newtheorem{theorem}{Theorem}[section]
\newtheorem{lemma}[theorem]{Lemma}
\newtheorem{corollary}[theorem]{Corollary}
\newtheorem{proposition}[theorem]{Proposition}

\theoremstyle{remark}
\newtheorem{example}{Example}[section]
\newtheorem{remark}{Remark}[section]
\newtheorem{assumption}{Assumption}[section]
\theoremstyle{definition}
\newtheorem{definition}{Definition}[section]

\newcommand{\const}{{\mathop{\rm const}\nolimits}}

\newcommand{\myp}{\mbox{$\:\!$}}
\newcommand{\mypp}{\mbox{$\;\!$}}
\newcommand{\myn}{\mbox{$\;\!\!$}}
\newcommand{\mynn}{\mbox{$\:\!\!$}}

\newcommand{\PP}{P}
\newcommand{\QQ}{Q}

\newcommand{\EE}{E}
\newcommand{\Var}{\mathrm{Var}_z}

\newcommand{\supp}{\mathop{\rm supp}\nolimits}

\newcommand{\Li}{\operatorname{Li}\nolimits}

\newcommand{\NN}{\mathbb{N}}
\newcommand{\ZZ}{\mathbb{Z}}
\newcommand{\RR}{\mathbb{R}}
\newcommand{\CC}{{\mathbb C}}

\newcommand{\CP}{\varLambda}

\newcommand{\rme}{{\mathrm e}}
\newcommand{\rmi}{{\mathrm i}}
\newcommand{\dif}{{\mathrm d}}

\begin{document}
%
\begin{center}
{\LARGE Unified Derivation of the Limit Shape for Multiplicative
Ensembles of Random Integer Partitions with Equiweighted Parts

}

\vspace{1.1pc} {\large Leonid V.\ Bogachev}
\\[1pc]
\textit{Department of Statistics, School of Mathematics, University
of Leeds, Leeds LS2 9JT, UK.}\\
\textit{E-mail:} {\tt L.V.Bogachev@leeds.ac.uk}
\end{center}

\vspace{1pc} \centerline{Dedicated to Professor Anatoly M.~Vershik
on the occasion of his 80th birthday}

\medskip\noindent
\begin{abstract}
We derive the limit shape of Young diagrams, associated with growing
integer partitions, with respect to multiplicative probability
measures underpinned by the generating functions of the form
$\mathcal{F}(z)=\prod_{\ell=1}^\infty \myn\mathcal{F}_0(z^\ell)$
(which entails equal weighting among possible
parts~$\ell\in\NN$\myp). Under mild technical assumptions on the
function $H_0(u)=\ln\myn(\mathcal{F}_0(u))$, we show that the limit
shape $\omega^*(x)$ exists and is given by the equation
$y=\gamma^{-1}H_0(\rme^{-\gamma x})$, where $\gamma^2=\int_0^1
u^{-1}H_0(u)\,\dif{u}$. The wide class of partition measures covered
by this result includes (but is not limited to) representatives of
the three meta-types of decomposable combinatorial structures ---
assemblies, multisets, and selections. Our method is based on the
usual randomization and conditioning; to this end, a suitable local
limit theorem is proved. The proofs are greatly facilitated by
working with the cumulants of sums of the part counts rather than
with their moments.


%
\end{abstract}

\tableofcontents

 \vspace{1.5mm}

\normalfont

\section{Introduction}\label{sec1}

\subsection{Integer partitions and the limit shape problem}\label{sec1.1}
An \emph{integer partition} is a decomposition of a given natural
number into an \emph{unordered sum} of integers; for example,
$12=4+2+2+2+1+1$. More formally, a collection of integers
$\lambda=\{\lambda_1\ge\lambda_2\ge\dots>0,\:\lambda_i\in\NN\}$ is a
partition of $n\in\NN$ if $n=\lambda_1+\lambda_2+\cdots$, which is
sometimes written as $\lambda\vdash n$. We denote by $\CP_{n}$ the
(finite) set of partitions $\lambda\vdash n\in\NN$, and by
$\CP:=\cup_n\CP_n$ the collection of \emph{all} integer partitions.
The terms $\lambda_i\in\lambda$ are called \emph{parts} of the
partition $\lambda$. The alternative notation $\lambda=(1^{\nu_1}
2^{\nu_2}\mynn\dots)$ specifies the \emph{multiplicities} (or
\emph{counts}) of the parts involved,
$\nu_\ell:=\#\{\lambda_i\in\lambda\colon \lambda_i=\ell\}$
\,($\ell\in\NN$), with zero counts usually omitted from the
notation. (Here and below, $\vphantom{\int^y}\#\{\cdot\}$ denotes
the number of elements in a set.) It is evident that the part counts
satisfy the condition $\sum_{\ell=1}^\infty \ell\myp \nu_\ell=n$ for
any partition $\lambda=(1^{\nu_1} 2^{\nu_2}\mynn\dots)\in\CP_n$\myp.

A partition $\lambda=(\lambda_1,\lambda_2,\dots)$ is succinctly
visualized by its \emph{Young diagram} $\Upsilon_{\myn\lambda}$
formed by (left- and bottom-aligned) row blocks with
$\lambda_1,\lambda_2,\dots$ unit square cells (see
Fig.\mypp\ref{fig1}a). If $\lambda\in\CP_n$ (i.e., $\lambda\vdash
n$) then the area of the Young diagram $\Upsilon_{\myn\lambda}$
equals $n$. The upper boundary of $\Upsilon_{\myn\lambda}$ is a
piecewise-constant function
$Y_\lambda\colon[\myp0,\infty)\to\ZZ_+\myn:=\{0,1,2,\dots\}$ given
by (see Fig.\mypp\ref{fig1}b)
\begin{equation}\label{eq:Young}
Y_\lambda(x):=\sum_{\ell\ge x} \nu_\ell\myp,\qquad
\lambda=(1^{\nu_1}2^{\nu_2}\mynn\dots)\in\CP\myp.
\end{equation}
In particular, $Y_\lambda(0)=\sum_{\ell\ge0}
\mynn\nu_\ell=\#\{\lambda_i\in\lambda\}$ is the total number of
parts in partition $\lambda\in\CP$.

If the space $\CP_n$ is endowed with a probability measure $P_n$
(e.g., the uniform measure whereby all $\lambda\in\CP_n$ are
equiprobable) then one can speak of \emph{random partitions}
$\lambda\vdash n$. The \emph{limit shape}, with respect to a family
of probability measures $\PP_n$ on $\CP_n$ as $n\to\infty$, is
understood as (the graph of) a function $y=\omega^*(x)$ such that,
for every $\delta>0$ and any $\varepsilon>0$,
\begin{equation}\label{eq:LLN}
\lim_{n\to\infty}\PP_n\bigl\{\lambda\in\CP_n\colon
\sup\nolimits_{x\ge
\delta}\myn\bigl|\widetilde{Y}^n_\lambda(x)-\omega^*(x)\bigr|>\varepsilon\bigr\}=0,
\end{equation}
where $\widetilde{Y}^{n}_\lambda(x)=A_n^{-1} Y_\lambda(xB_n)$ for
suitable scaling constants $A_n,\myp{}B_n$. It is natural to require
that $A_nB_n=n$, which would render the area of the scaled Young
diagram
$\widetilde{\Upsilon}\vphantom{Y}^{n\vphantom{._.}}_{\myn\lambda}$
to be normalized to unity; the most frequent choice is specified as
$A_n=B_n=\sqrt{n}$\mypp.

\begin{center}
\begin{figure}[ht]
\thinlines
\mbox{}\hspace{1.5pc}\includegraphics[width=2.50in]{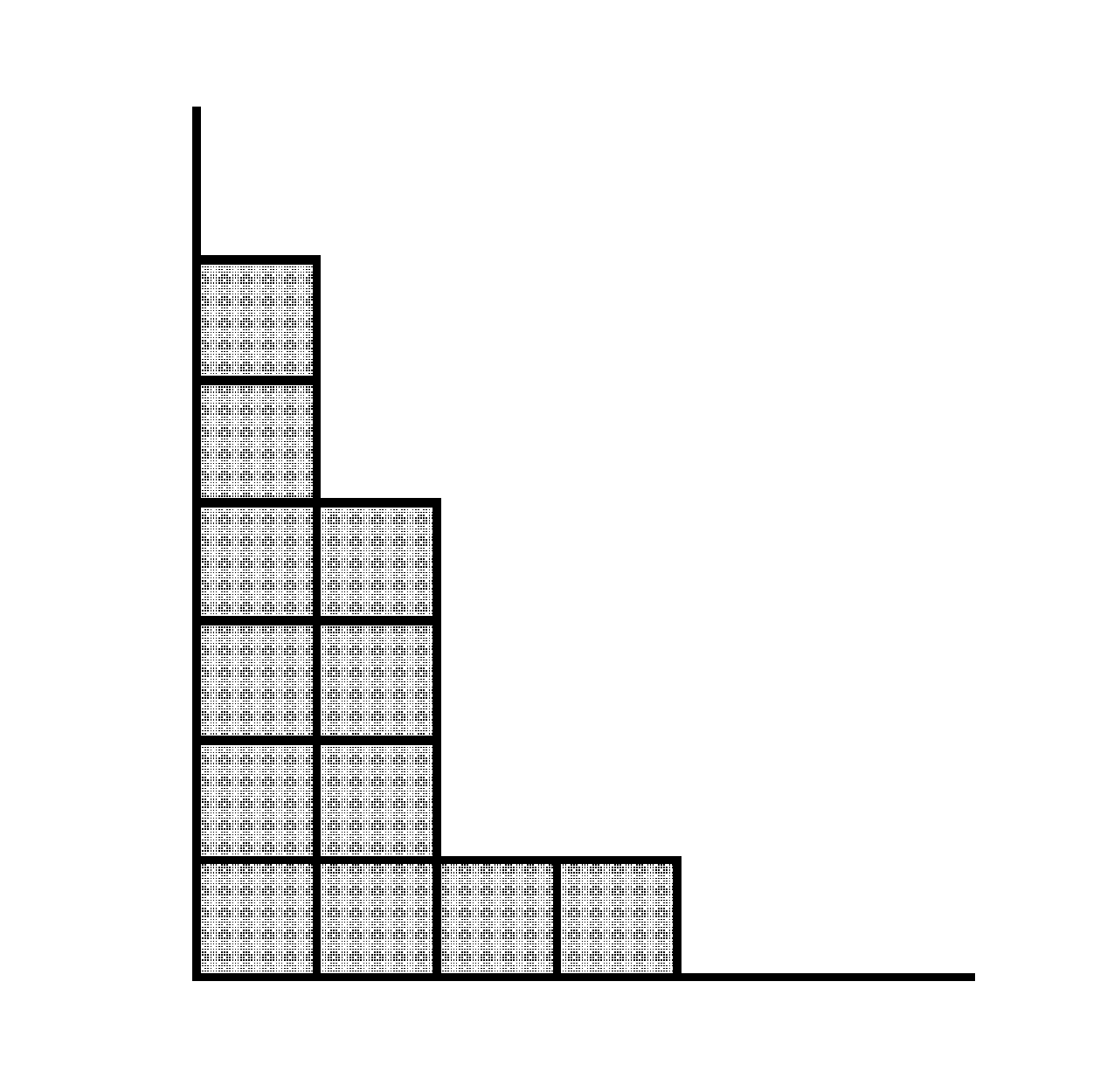}\hspace{.5in}%
\includegraphics[width=2.50in]{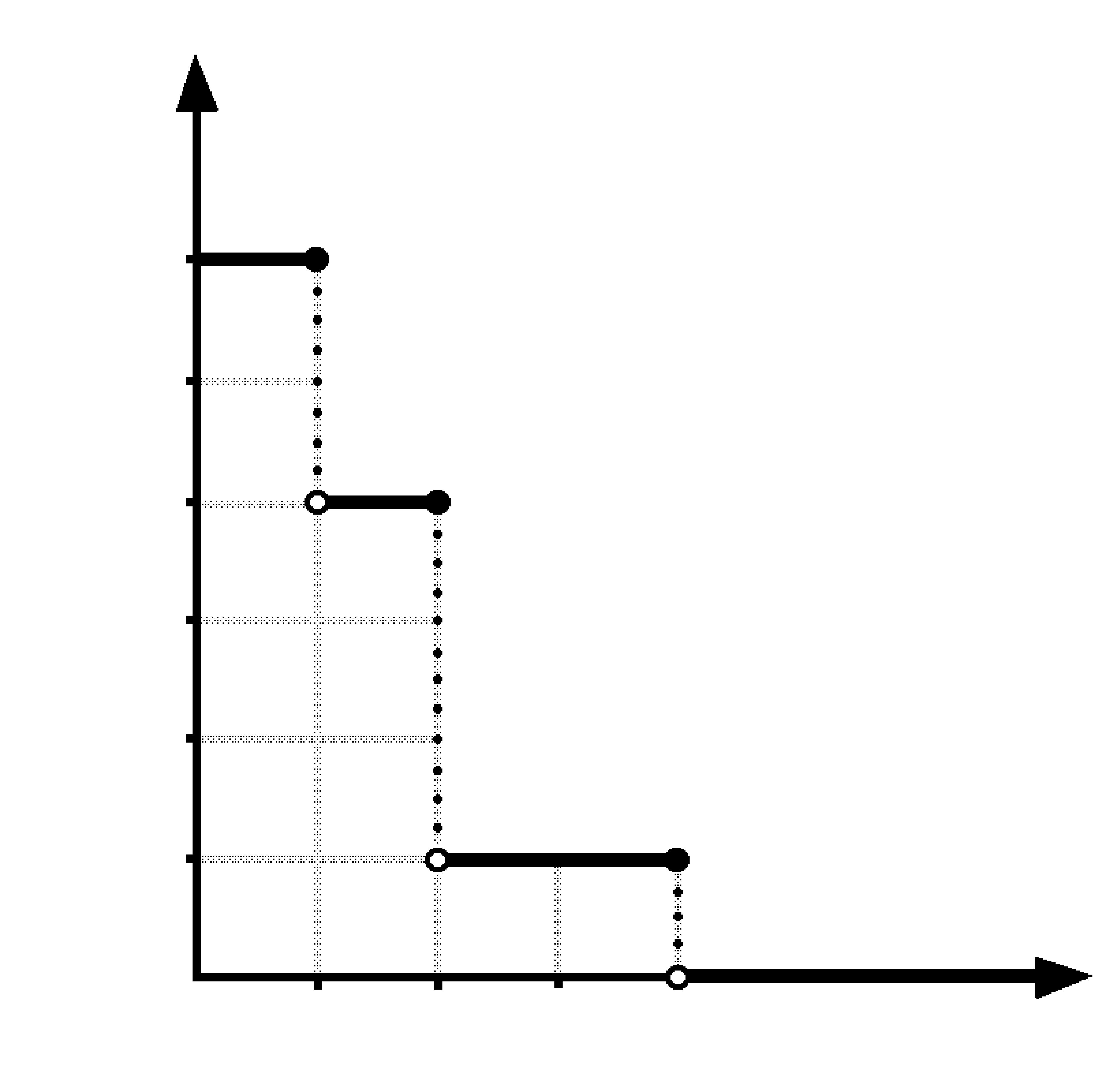}
\put(-322,-7){\mbox{\small (a)}} %
\put(-87,-6){\mbox{\small (b)}}  %
\put(-5,8){\mbox{\footnotesize$x$}}
\put(-330,160){\mbox{\small$\Upsilon_{\myn\lambda}$}}
\put(-97,160){\mbox{\small$Y_\lambda(x)$}}
\put(-160,161){\mbox{\footnotesize$y$}}
\put(-369,8){\mbox{\scriptsize$0$}}
\put(-348.2,8){\mbox{\scriptsize$1$}}
\put(-329,8){\mbox{\scriptsize$2$}}
\put(-309.5,8){\mbox{\scriptsize$3$}}
\put(-290,8){\mbox{\scriptsize$4$}}
\put(-152,8){\mbox{\scriptsize$0$}}
\put(-131.5,8){\mbox{\scriptsize$1$}}
\put(-112,8){\mbox{\scriptsize$2$}}
\put(-92,8){\mbox{\scriptsize$3$}}
\put(-72.5,8){\mbox{\scriptsize$4$}}
\put(-374,35){\mbox{\scriptsize$1$}}
\put(-374,55){\mbox{\scriptsize$2$}}
\put(-374,74){\mbox{\scriptsize$3$}}
\put(-374,93){\mbox{\scriptsize$4$}}
\put(-374,113){\mbox{\scriptsize$5$}}
\put(-374,133){\mbox{\scriptsize$6$}}
\put(-158,35){\mbox{\scriptsize$1$}}
\put(-158,55){\mbox{\scriptsize$2$}}
\put(-158,74){\mbox{\scriptsize$3$}}
\put(-158,93){\mbox{\scriptsize$4$}}
\put(-158,113){\mbox{\scriptsize$5$}}
\put(-158,133){\mbox{\scriptsize$6$}}
\put(-304,89.5){\footnotesize$\lambda_4\!=\!2$} %
\put(-344,124){{\line(3,1){14}}}
\put(-328,126){\footnotesize$\lambda_6\!=\!1$} %
\put(-344,105){{\line(3,1){14}}} %
\put(-328,107){\footnotesize$\lambda_5\!=\!1$} %
\put(-325,86){{\line(3,1){19}}} \put(-325,67){{\line(3,1){19}}}
\put(-304,70.5){\footnotesize$\lambda_3\!=\!2$} %
\put(-325,48){{\line(3,1){19}}}
\put(-304,51.5){\footnotesize$\lambda_2\!=\!2$} %
\put(-286,28){{\line(3,1){19}}}
\put(-265,31.5){\footnotesize$\lambda_1\mynn\!=\!4$} %
\caption{The Young diagram $\Upsilon_{\myn\lambda}$ (a) \,and the
graph of its upper boundary $Y_\lambda(x)=\sum_{\ell\ge x}
\mynn\nu_\ell$ \,(b) \,for a partition
$\lambda=(4,2,2,2,1,1)\equiv(1^2 \mypp 2^3\myp 4^1)\vdash n=12$,
with the part counts $\nu_1=2$, $\nu_2=3$, and $\nu_4=1$.}
\label{fig1}
\end{figure}
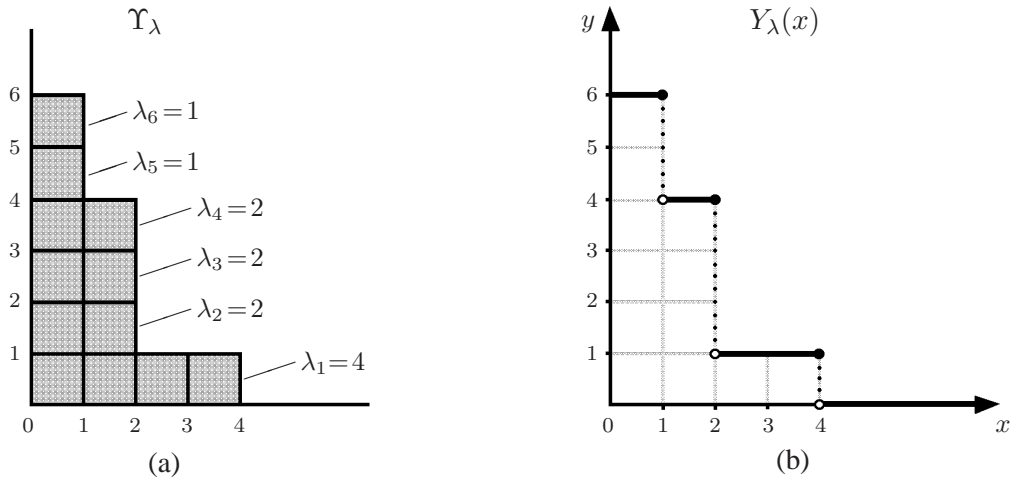
\end{center}

\vspace{-1.06pc} Of course, the limit shape and its very existence
depend on the chosen family of probability laws $\PP_n$ on the
partition spaces $\CP_n$ \myp($n\in\NN$). With respect to the
uniform (equiprobable) distribution on $\CP_n$\myp, the limit shape
$\omega^*(x)$ exists under the scaling $A_n=B_n=\sqrt{n}$ and is
determined by the equation (see Fig.\,\ref{fig2}a)
\begin{equation}\label{eq:limit1}
\rme^{-x\pi/\sqrt{6}}+\rme^{-y\pi/\sqrt{6}}=1,\qquad x,y\ge0.
\end{equation}
The limit shape \eqref{eq:limit1} was first identified by Temperley
\cite{Temperley} in relation to the equilibrium shape of a growing
crystal, and derived more rigorously much later by Vershik (as noted
in \cite[p.\:30]{VK1985}) using some asymptotic estimates from
Szalay and Tur\'an \cite{ST}. The proof in its modern form was
outlined by Vershik in \cite{V3}; an alternative proof was given by
Pittel~\cite{Pittel}.

Unlike \cite{Pittel} where only the uniform case was studied,
Vershik's method was used in \cite{V3} to settle the limit shape
problem for more general partition ensembles of the so-called
\emph{multiplicative type} (see Section \ref{sec1.2} below),
including the uniform distribution on the subset
$\check{\CP}_n\subset \CP_n$ of \emph{strict} partitions (i.e., with
distinct parts, $\nu_\ell\le 1$ for all $\ell\in\NN$), whereby the
limit shape, under the same scaling, appears to be of the form (see
Fig.\,\ref{fig2}b)
\begin{equation}\label{eq:limit2}
\rme^{\myp y\pi/\sqrt{12}}=1+\rme^{-x\pi/\sqrt{12}},\qquad x,y\ge0.
\end{equation}
\begin{center}
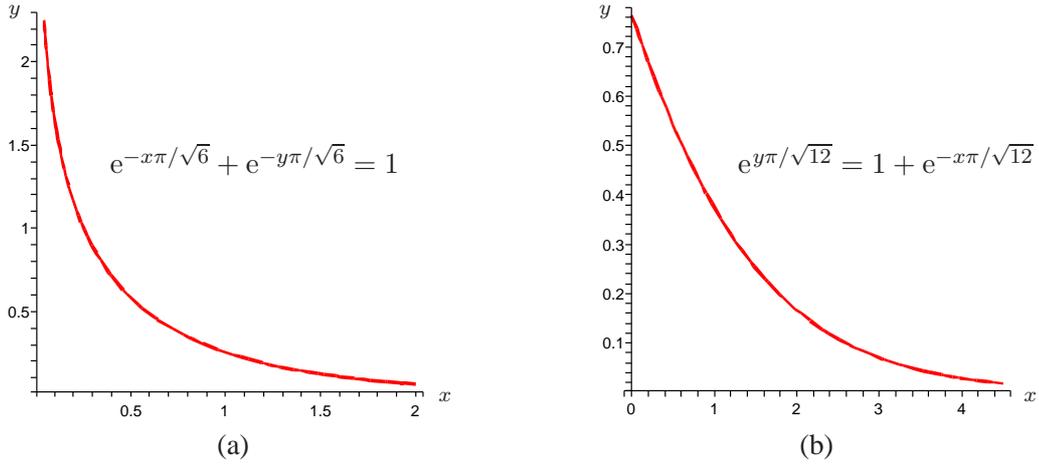
\begin{figure}[ht]
\thinlines
\mbox{}\hspace{2.5pc}\includegraphics[width=2.50in]{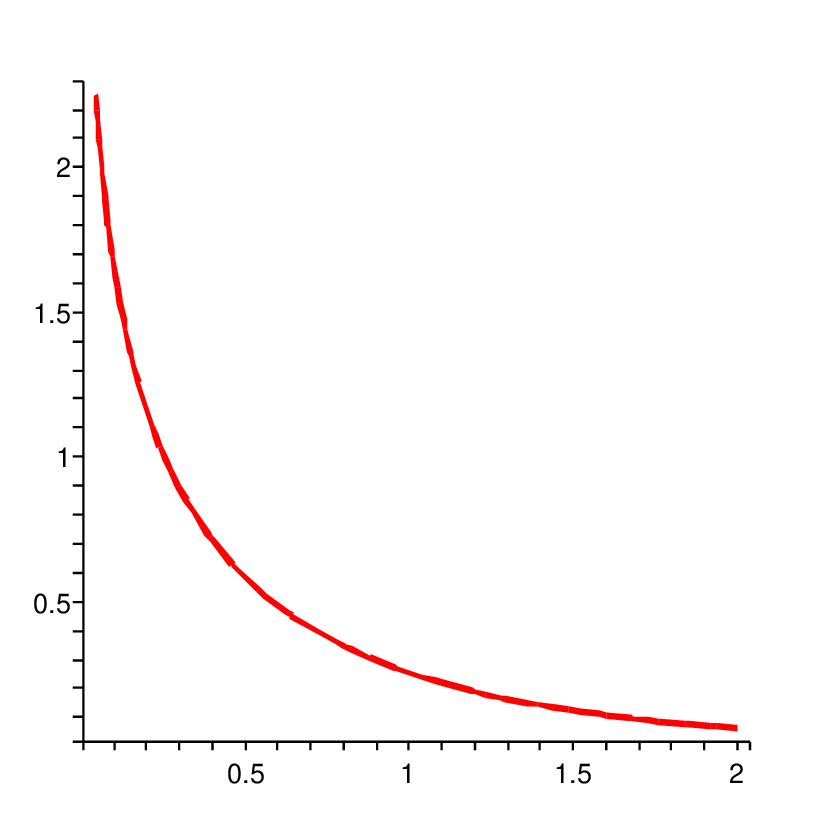}\hspace{.5in}
\includegraphics[width=2.50in]{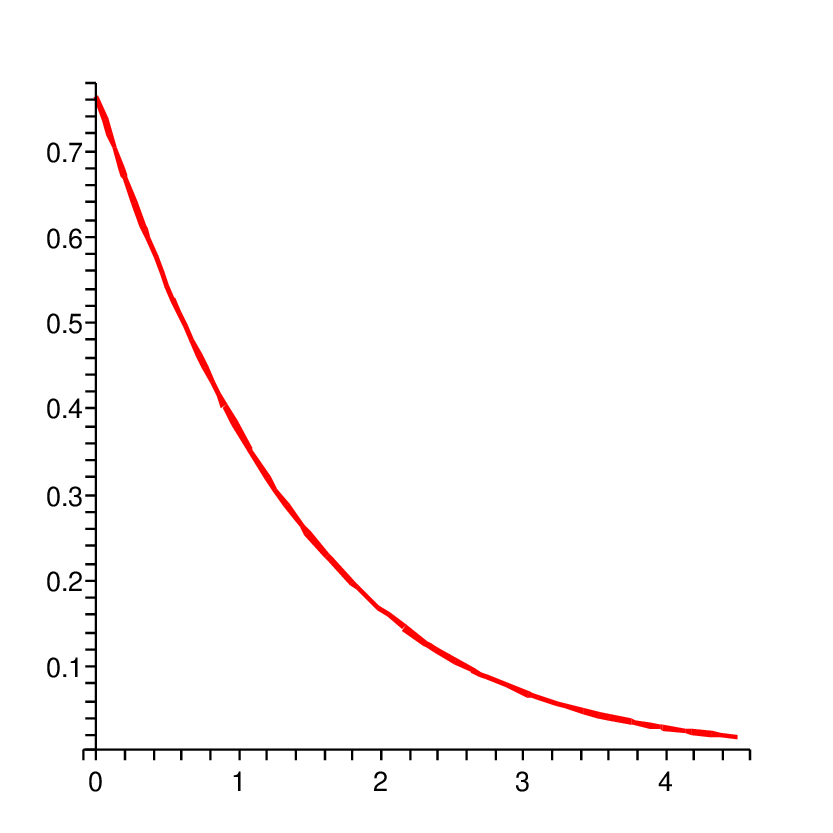}
\put(-355,100){\mbox{\small $\rme^{-x\pi/\sqrt{6}}+\rme^{-y\pi/\sqrt{6}}=1$}} %
\put(-120,100){\mbox{\small $\rme^{\myp y\pi/\sqrt{12}}=1+\rme^{-x\pi/\sqrt{12}}$}} %
\put(-315,-6){\mbox{\small (a)}} %
\put(-97,-6){\mbox{\small (b)}}  %
\put(-232,14){\mbox{\scriptsize$x$}}
\put(-13,14){\mbox{\scriptsize$x$}}
\put(-393,160){\mbox{\scriptsize$y$}}
\put(-172,160){\mbox{\scriptsize$y$}} %
\caption{The limit shape $y=\omega^*(x)$ for two classical ensembles
of uniform (equiprobable) random partitions: (a) unrestricted
partitions ($\CP_n$); \,(b) partitions with distinct parts
($\check{\CP}_n$). In both cases, the normalization in
\eqref{eq:LLN} is specified by $A_n=B_n=\sqrt{n}$\mypp.}
\label{fig2}
\end{figure}
\end{center}

\subsection{Multiplicative measures on partition spaces}\label{sec1.2}
For a general discussion and plentiful examples of multiplicative
probability measures on partitions, the reader may consult the
classic work by Vershik \cite{V3,V4} and more recent papers by
Erlihson and Granovsky~\cite{EG}, Su \cite{Su1} and Yakubovich
\cite{Yakubovich}, with an abundance of further references therein.
In the monograph by Arratia et al.\
\cite{ABT}, such measures are considered in the general context of
decomposable combinatorial structures.

In short, multiplicative measures are underpinned by the generating
functions of the form
\begin{equation}\label{eq:gfF}
\mathcal{F}(z)=\prod_{\ell=1}^\infty
\mathcal{F}_{\myn\ell\myp}(z^\ell)=\prod_{\ell=1}^\infty
\sum_{k=0}^\infty c_k^{(\ell)}z^{k\ell}, \quad\text{with\ \
$c_0^{(\ell)}\equiv1$,\ \ $c_k^{(\ell)}\ge0$\ \ \,($k,\ell\in\NN$)}.
\end{equation}

\vspace{-.5pc} \noindent More precisely, the corresponding family of
the measures $\PP_n$ on the respective partition spaces $\CP_n$
\mypp($n\in\NN$) is defined by setting
\begin{equation}\label{eq:b-Gamma-mult}
\PP_n(\lambda):=\mathfrak{C}_n^{-1}\prod_{\ell=1}^\infty \myp
c_{\nu_\ell}^{(\ell)}\myp,\qquad
\lambda=(1^{\nu_1}2^{\nu_2}\mynn\dots)\in\CP_n\myp,
\end{equation}
where $\mathfrak{C}_n$ is the suitable normalization constant. For
example, the generating function
$\mathcal{F}_{\myn\ell\myp}(u)\equiv (1-u)^{-1}=\sum_{k=0}^\infty
u^k$ defines the uniform measure on each $\CP_n$\myp, whereas the
choice $\mathcal{F}_{\myn\ell\myp}(u)\equiv 1+u$ leads to the
uniform measure on the space $\check{\CP}_n$ of strict partitions.

According to \eqref{eq:b-Gamma-mult}, each generating function
$\mathcal{F}_{\myn\ell\myp}(\cdot)$ assigns some weights, relative
to the uniform case with $c^{(\ell)}_k\equiv 1$, to specific values
of the part count $\nu_{\ell}=\#\{\lambda_i=\ell\}$ in a random
partition $\lambda=(1^{\nu_1}2^{\nu_2}\mynn\dots)\in\CP$.
Furthermore, possible variation of the functions
$\mathcal{F}_{\myn\ell\myp}(\cdot)$ across $\ell\in\NN$ determines a
certain weighting among different parts that may contribute to a
partition.
\begin{definition}\label{def:equiweighted}
If the functions $\mathcal{F}_{\myn\ell\myp}(\cdot)$ do not depend
on $\ell$ (hence, $c^{(\ell)}_k\equiv c_k$ for all $\ell\in\NN$)
then we say that the parts are \emph{equiweighted} (which is alluded
to in the title of the paper).
\end{definition}

\begin{remark}
Note from the definition \eqref{eq:b-Gamma-mult} that the marginal
distribution of a random count $\nu_\ell$ is \emph{$\ell$-biased},
being given by $\PP_n\{\nu_\ell=k\}=c^{(\ell)}_k
\mathfrak{C}^{(\ell)}_{n-k\ell}\myp/\mathfrak{C}_n$ ($0\le k\le
n/\ell$), where
$\mathfrak{C}^{(\ell)}_{m}:=\sum_{\lambda\in\CP_{m}\!}\prod_{j\ne
\ell} c_{\nu_j}^{(j)}$ (with the convention
$\mathfrak{C}^{(\ell)}_{0}\mynn:=1)$. Thus, the assumption that the
parts are equiweighted\strut{} does not imply that their counts have
the same distribution.
\end{remark}

Building on Vershik's pioneering ideas, the limit shape problem was
advanced in various directions (see
\cite{Comtet,DVZ,EG,GH,GSE,Romik,Su1,V4,VY-MMJ,Yakubovich} and
further references therein). In a separate but related development,
Logan and Shepp \cite{LS1977} and Vershik and Kerov
\cite{VK1977,VK1985} found the limit shape for a different
(non-multiplicative) ensemble of partitions endowed with the
\emph{Plancherel measure} emerging in relation with representation
theory of the symmetric group. A recent review of both areas can be
found in~\cite{Su1}.

Returning to the multiplicative class of probability measures on
partitions, note that most of the aforementioned papers on the limit
shape problem have focused on the particular case
\begin{equation}\label{eq:b_ell}
\mathcal{F}_{\myn\ell\myp}(u)=
(\mathcal{F}_0(u))^{r_\ell}\myn,\qquad \ell\in\NN,
\end{equation}
for some classes of sequences $r_\ell>0$ (usually assumed to behave
like $r_\ell\sim \const\cdot \ell^{\myp p-1}$ as $\ell\to\infty$,
with $p>0$) but subject to a more limited choice of the basic
generating function $\mathcal{F}_0(u)$, often borrowed from the
standard equiprobable cases mentioned above (see, e.g.,
\cite{V3,V4,EG, GSE,Su1}).

A recent paper by Yakubovich \cite{Yakubovich} offers a more general
treatment by considering a wider class of functions
$\mathcal{F}_0(u)$; a typical condition imposed there (see, e.g.,
\cite[Lemma 10]{Yakubovich}) is that $\mathcal{F}_0(u)$ be complex
analytic in a disk centered at zero up to an isolated (real)
singularity point $u_1\ge 1$, which \emph{must be a pole} if
$u_1=1$. Some simple examples such as $\mathcal{F}_0(u)=(1-u)^{-r}$
with a real (non-integer) $r>0$ do not formally conform to this
requirement but none the less have a limit shape \cite{V3}
(the assumption in \cite{V3} that $r_\ell$'s are integer is in fact
not essential in the light of the Meinardus theorem, see
\cite{GSE}). On the other hand, one can write
$\mathcal{F}_0(u)=(f_0(u))^r$, where the function
$f_0(u)=(1-u)^{-1}$ has a required pole at $u_1=1$ and thus fits in
the framework of \cite{Yakubovich}.\footnote{\ This substitution
replaces the normalization $r_1=1$ adopted for convenience in
\cite[\S1.1, p.\,1254]{Yakubovich} by $r_1=r>0$, which is not
essential for the validity of results in \cite{Yakubovich}.
Incidentally, this remark shows that it is more natural to impose
conditions on the function $H_0(u):=\ln\myn(\mathcal{F}_0(u))$
rather than on $\mathcal{F}_0(u)$ itself.\label{footnote1}}\
\,However, there are examples with a genuine non-pole singularity of
$\mathcal{F}_0(u)$ which do possess a limit shape (see such examples
in Section~\ref{sec6} below).

\subsection{An outline of the main result}

In the present paper, we confine ourselves to the class of
multiplicative ensembles of partitions \emph{with equiweighted
parts} (see Definition~\ref{def:equiweighted}), specified by the
simplest case $\mathcal{F}_{\myn\ell}(u)\equiv\mathcal{F}_0(u)$ in
\eqref{eq:gfF} (which also corresponds to setting $r_\ell\equiv 1$
in the model \eqref{eq:b_ell}) but with a fairly general variety of
permissible generating functions $\mathcal{F}_0(u)$. In particular,
measures $\PP_n$ covered by our method include (but are not limited
to) representatives of the three classical meta-types of
decomposable combinatorial structures
--- assemblies, multisets and selections
(see \cite[Ch.\:2]{ABT} for a general background and also concrete
examples in Section \ref{sec6} below).

A loose formulation of our main result about the limit shape is as
follows.

\begin{theorem}\label{th:main1}
Denote $H_0(u):=\ln\myn(\mathcal{F}_0(u))$,
\,$\gamma:=\sqrt{\int_0^1\myn u^{-1}H_0(u)\,\dif{u}}$\,, \,and
\begin{equation}\label{eq:omega}
\omega^*(x):=\gamma^{-1}H_0(\rme^{-\gamma x}), \qquad x\ge0.
\end{equation}
Then, under mild technical conditions on the function $H_0(u)$, for
every\/ \myp$\delta>0$ and any\/ $\varepsilon>0$
\begin{equation}\label{eq:looseLS}
\lim_{n\to\infty}\PP_n\bigl\{\lambda\in\CP_n\colon\sup\nolimits_{x\ge\delta}
\mynn\bigl|\widetilde{Y}^n_\lambda(x)-\omega^*(x)\bigr|
>\varepsilon\bigr\}=0,
\end{equation}
where $\widetilde{Y}^n_\lambda(x):=n^{-1/2}\myp Y_\lambda(x\myp
n^{1/2})$.
\end{theorem}

\begin{remark}\label{rm:1.3}
The restriction $x\ge\delta>0$ in \eqref{eq:looseLS} takes into
account the possibility $\omega^*(0)=\infty$ (cf.\
\eqref{eq:limit1},~\eqref{eq:limit2}). If $\omega^*(0)<\infty$ then
the supremum in \eqref{eq:looseLS} can be extended to all $x\ge0$.
\end{remark}

Like in \cite{Fristedt,V3,V4,EG,Yakubovich}, our proof employs the
elegant probabilistic approach in the theory of decomposable
combinatorial structures based on randomization and conditioning,
first applied in the context of random partitions by Fristedt
\cite{Fristedt} (see the monograph \cite{ABT} and an earlier review
\cite{AT} for a general discussion of the method and many examples).
The idea is to introduce a suitable measure $\QQ_z$ on the union
space $\CP=\cup_n\CP_n$ (depending on an auxiliary ``free''
parameter $z\in(0,1)$), such that a given measure $\PP_n$ on $\CP_n$
is recovered as the conditional distribution
$\PP_n(\cdot)=\QQ_z(\cdot\mypp|\myp\CP_n)$.

The great advantage of the multiplicativity property \eqref{eq:gfF}
is that $\QQ_z$ can be constructed as a product measure, resulting
in \emph{mutually independent} random counts $\nu_\ell$\myp.
Clearly, such a device calls for the asymptotics of the probability
$\QQ_z(\CP_n)$, which should be obtained by proving a suitable
\emph{local limit theorem}; the latter suggests that it is natural
to calibrate the parameter $z$ from the asymptotic equation
$\EE_z(N_\lambda)=n\left(1+o(1)\right)$, where $\EE_z$ is the
expectation with respect to the measure $\QQ_z$ and
$N_\lambda:=\lambda_1+\lambda_2+\cdots=\sum_{\ell=1}^\infty
\ell\myp\nu_\ell$ \,(so that $\CP_n=\{\lambda\in\CP\colon
N_\lambda=n\}$). This is sufficient to ensure the (uniform)
convergence of the \emph{expectation}
$\EE_z\bigl[\widetilde{Y}^n_\lambda(x)\bigr]$ to the limit
$\omega^*(x)$ specified in \eqref{eq:omega}, together with the
corresponding convergence of the random paths
$\widetilde{Y}^n_\lambda(\cdot)$ in $Q_z$-probability. However, in
order to extend this to the original measure $\PP_n$ using the local
limit theorem, our methods require an improved estimate of the
approximation error $\EE_z(N_\lambda)-n$ to at least $O(n^{3/4})$.
Let us also point out that the proofs are greatly facilitated by
working with the \emph{cumulants} of sums of the part counts
$\nu_\ell$ rather than with their moments.

\subsubsection*{Layout.} The rest of the paper is organized as
follows. In Section~\ref{sec2.1}, we define the multiplicative
families of measures $\QQ_z$ and $\PP_n$ on the corresponding spaces
of partitions with equiweighted parts. Important cumulant expansions
and certain technical conditions on the generating function
$\mathcal{F}_0(u)$ are discussed in Section \ref{sec2.2}. In
Section~\ref{sec3.1}, a suitable value of the parameter $z\in(0,1)$
is chosen (Theorem~\ref{th:kappa}), which implies the convergence of
``expected'' (scaled) Young diagrams to the limit curve
$y=\omega^*(x)$ (Theorem~\ref{th:3.2}). Refined first-order moment
asymptotics are obtained in Section~\ref{sec3.3}
(Theorem~\ref{th:4.1}), while higher-order cumulant sums are
analyzed in Section~\ref{sec4}. The local limit theorem
(Theorem~\ref{th:LCLT}) is established in Section \ref{sec5}, which
paves the way to the proof of the limit shape results in
Section~\ref{sec5.4} with respect to both $\QQ_z$ and $\PP_n$
(Theorems \ref{th:LSQ} and~\ref{th:LSP}, respectively). Finally, our
results are illustrated by a number of examples in
Section~\ref{sec6}.

\subsubsection*{Some notation.}
We denote $\ZZ_+:=\{k\in\ZZ\colon k\ge0\}$ and
$\RR_+:=[\myp0,\infty)$. The real part of $s\in\CC$ is denoted
$\Re\myp(s)$. The notation $x_n\asymp y_n$ signifies that
$0<\liminf_{n\to\infty} x_n/y_n\le\limsup_{n\to\infty}
x_n/y_n<\infty$, whereas $x_n\sim y_n$ is a shorthand for
$\lim_{n\to\infty}x_n/y_n=1$. The standard symbols $\lfloor
x\rfloor:=\max\myn\{k\in\ZZ\colon k\le x\}$ and $\lceil
x\rceil:=\min\myn\{k\in\ZZ\colon k\ge x\}$ denote, respectively, the
floor and ceiling integer parts of $x\in\RR$\myp.

\section{Generating functions and cumulants}\label{sec2}

\subsection{Global measure\/ \myp$\QQ_{z}$ and conditional
measure\/ $\PP_n$}\label{sec2.1}

Let $\Phi:= \ZZ_+^{\NN}$ be the space of functions $\nu\colon
\NN\to\ZZ_+$ (i.e., sequences $\nu=\{\nu_\ell\}$ with nonnegative
integer values), and consider the subspace $\Phi_0
:=\{\nu\in\Phi\,{:}\;\#(\supp\myp\nu)<\infty\}$ of functions with
\emph{finite support}, where $\supp\myp\nu:=\{\ell \in
\NN\,{:}\;\nu_\ell >0\}$. The space $\Phi_0$ is in one-to-one
correspondence with the union set $\CP=\bigcup_{n\in\ZZ_+}\!\CP_n$
under the identification \strut{}of the values $\nu_\ell$'s
(including zeroes) with the multiplicities of the virtual parts
$\ell$'s, respectively, leading to a partition
$\lambda=(1^{\nu_1}2^{\nu_2}\mynn\dots)$ of the integer
$N_\lambda=\sum_{\ell=1}^\infty \ell\myp \nu_\ell \in\ZZ_+$\myp.

Let $c_0=1,c_1,c_2,\dots$ be a sequence of nonnegative numbers such
that \emph{not all $c_{k}$'s vanish for $k\ge1$}, and assume that
the corresponding power series (generating function)
\begin{equation}\label{eq:F0}
\mathcal{F}_0(u):=\sum_{k=0}^\infty c_{k}\myp u^{k},\qquad u\in\CC,
\end{equation}
is convergent for all $|u|<1$. For every $z\in(0,1)$, let us define
a probability measure $\QQ_{z}$ on the space $\Phi=\ZZ_+^{\NN}$ as
the distribution of a random sequence $\{\nu_\ell,\,\ell\in \NN\}$
with \emph{mutually independent values} and \emph{marginal
distributions}
\begin{equation}\label{Q}
\QQ_{z}\{\nu_\ell=k\}=\frac{c_k
z^{k\ell}}{\mathcal{F}_0(z^\ell)}\myp,\qquad k\in\ZZ_+\myp.
\end{equation}

\begin{lemma}\label{lm:F0}
For $z\in (0,1)$, the condition
\begin{equation}\label{N}
\mathcal{F}(z):=\prod_{\ell=1}^\infty \mathcal{F}_0(z^{\ell})<\infty
\end{equation}
is necessary and sufficient in order that $\QQ_{z}(\Phi_0)=1$.
Furthermore, if\/ $\mathcal{F}_0(u)$ is finite for all $u\in(0,1)$
then the condition \eqref{N} is satisfied for all $z\in(0,1)$.
\end{lemma}
\begin{proof}
According to \eqref{Q} we have
$\QQ_{z}\{\nu_\ell>0\}=1-1/\mathcal{F}_0(z^\ell)$
\mypp($\ell\in\NN$). Hence, Borel--Cantelli's lemma (see, e.g.,
\cite[Ch.\,VIII, \S\myp3]{Feller1}) implies that
$\QQ_{z}\{\nu\in\Phi_0\}=1$ if and only if $\sum_{\ell=1}^\infty
\left(1-1/\mathcal{F}_0(z^\ell)\right)<\infty$. In turn, the latter
bound is equivalent to~\eqref{N}.

To prove the second statement, observe using \eqref{eq:F0} that
\begin{align*}
\ln\myn(\mathcal{F}(z))=\sum_{\ell=1}^\infty\ln\myn(\mathcal{F}_0(z^\ell))
&\le \sum_{\ell=1}^\infty
\bigl(\mathcal{F}_0(z^\ell)-1\bigr)=\sum_{k=1}^\infty
c_k \sum_{\ell=1}^\infty z^{k \ell}\\
&=\sum_{k=1}^\infty \frac{c_kz^k}{1-z^k}\le
\frac{1}{1-z}\sum_{k=1}^\infty c_k z^k\le
\frac{\mathcal{F}_0(z)}{1-z}<\infty,
\end{align*}
which implies the condition~\eqref{N}.
\end{proof}

Lemma \ref{lm:F0} ensures that the random sequence $\{\nu_\ell\}$
defined above (see~\eqref{Q}) belongs to the space $\Phi_0$
($\QQ_{z}$-a.s.)\footnote{\ The abbreviation ``a.s.'' stands for
\emph{almost surely}, that is, with probability~$1$.} and therefore
determines a \emph{finite} (random) partition $\lambda\in \CP$. By
the mutual independence of the values $\nu_\ell$\myp, the
corresponding $\QQ_z$-probability is given by
\begin{equation}\label{Q1}
\QQ_{z}(\lambda) =\prod_{\ell=1}^\infty\frac{c_{\nu_\ell}\myp
z^{\ell\myp\nu_\ell}}{\mathcal{F}_0(z^\ell)} =\frac{c(\lambda)\myp
z^{N_\lambda}} {\mathcal{F}(z)}\myp,\qquad \lambda=(1^{\nu_1}
2^{\nu_2}\mynn\dots)\in\CP,
\end{equation}
where $N_\lambda=\sum_{\ell=1}^\infty \ell\myp\nu_\ell<\infty$
($\QQ_z$-a.s.) and (see~\eqref{eq:F0})
\begin{equation}\label{eq:c-mult}
c(\lambda)=\prod_{\ell=1}^\infty c_{\nu_\ell}<\infty,\qquad
\lambda=(1^{\nu_1} 2^{\nu_2}\mynn\dots)\in\CP.
\end{equation}

\begin{remark}
The infinite product \eqref{eq:c-mult} defining $c(\lambda)$
contains only finitely many factors different from $1$, because any
$\ell\notin\supp\myp\nu$ renders $\nu_\ell=0$, so that
$c_{\nu_\ell}=c_0=1$.
\end{remark}
\begin{remark}
For the ``empty'' partition $\lambda_{\emptyset}\vdash 0$ formally
associated with the configuration $\nu\equiv0$, formula \eqref{Q1}
yields $\QQ_z(\lambda_{\emptyset})=1/\mathcal{F}(z)>0$. On the other
hand, $\QQ_z(\lambda_{\emptyset})<1$, since
$\mathcal{F}_0(u)>\mathcal{F}_0(0)=1$ for $u>0$ and hence, according
to the definition~\eqref{N}, $\mathcal{F}(z)>1$.
\end{remark}

On the subspace $\CP_{n}\subset\CP$, the measure $\QQ_{z}$ induces
the conditional distribution
\begin{equation}\label{Pn}
\PP_n(\lambda):=\QQ_{z}(\lambda\myp|\myp\CP_n)
=\frac{\QQ_{z}(\lambda)}{\QQ_{z}(\CP_n)} \myp,\qquad \lambda\in
\CP_n\myp.
\end{equation}
The formula \eqref{Pn} is well defined as long as
$\QQ_{z}(\CP_n)>0$, that is, if there is at least one partition
$\lambda\in\CP_n$ with $c(\lambda)>0$ (see~\eqref{Q1}). An obvious
sufficient condition is as follows.
\begin{lemma}
Suppose that $c_1>0$. Then $\QQ_z(\CP_n)>0$ for all $n\in\ZZ_+$\myp.
\end{lemma}

The following key fact is a direct consequence of the
definition~\eqref{Q1}.

\begin{lemma}\label{lm:noz}
The formula\/ \eqref{Pn} for the measure\/ $\PP_n$ is reduced to the
expression \textup{(}cf.\ \eqref{eq:b-Gamma-mult}\textup{)}
\begin{equation}\label{eq:P_n}
\PP_n(\lambda)=\frac{c(\lambda)}{\mathfrak{C}_n}\ \ \quad
(\lambda\in\CP_n),\qquad \mathfrak{C}_n=\sum_{\lambda'\in\CP_n}
c(\lambda'),
\end{equation}
where $c(\lambda)$ is defined in \eqref{eq:c-mult}. In particular,
$P_n$ does not depend on~$z$\myp.
\end{lemma}

\begin{proof} If $\CP_n\ni\lambda\leftrightarrow\nu\in\Phi_0$ then
$N_\lambda=n$ and the formula \eqref{Q1} is reduced to
$\QQ_{z}(\lambda)= c(\lambda)\myp z^n/\mathcal{F}(z)$. In turn, the
ratio in \eqref{Pn} amounts to the expression in~\eqref{eq:P_n},
which is $z$-free.
\end{proof}

Specific examples of multiplicative measures $\PP_n$ with
equiweighted parts will be given below in Section~\ref{sec6},
together with the corresponding limit shapes determined by
Theorem~\ref{th:main1}.

\subsection{Expansion of the logarithm of the generating function $\mathcal{F}_0(u)$}
Recalling the power series expansion \eqref{eq:F0} for
$\mathcal{F}_0(u)$, consider the corresponding expansion of its
logarithm,
\begin{equation}\label{eq:H}
H_0(u):=\ln\myn(\mathcal{F}_0(u))=\sum_{k=1}^\infty a_k\myp
u^k,\qquad u\in\CC\myp,
\end{equation}
assuming that the series \eqref{eq:H} is (absolutely) convergent for
all $|u|<1$. Here $\ln\myn(\cdot)$ means the principal branch of the
logarithm specified by the value $\ln\myn(\mathcal{F}_0(0))=\ln
1=0$.

\begin{remark}\label{rm:b=a}
Substituting \eqref{eq:F0} into \eqref{eq:H}, it is evident that
$a_1=c_1$; more generally, if $j_*\myn:=\min\{j\ge 1\colon
a_j\ne0\}$ and $k_*\myn:=\min\{k\ge 1\colon c_{k}>0\}$ then
$j_*\myn=k_*$ and $a_{j_*}\mynn=c_{k_*}\mynn>0$. In particular, it
follows that the first non-vanishing coefficient in the power series
\eqref{eq:H} is \emph{positive}.
\end{remark}

Differentiating \eqref{eq:H}, we get the standard formulas for the
power sums
\begin{align}\label{eq:H'}
\sum_{k=1}^\infty k a_k\myp u^k &=u\myp H'_0(u),\\
\label{eq:H''} \sum_{k=1}^\infty k^2 a_k\myp u^k &=u\myp\bigl(u\myp
H'_0(u)\bigr)'=u^2 H''_0(u)+u\myp H'_0(u),
\end{align}
with similar expressions available for the higher-order sums
$\sum_{k=1}^\infty k^q a_k u^k$ \myp($q\in\NN$).

For $s\in\CC$ such that $\sigma:=\Re\myp(s)>0$, consider the
Dirichlet series
\begin{equation}\label{eq:A2-}
A(s):=\sum_{k=1}^\infty \frac{a_{k}}{k^s}\myp,\qquad
A^+(\sigma):=\sum_{k=1}^\infty \frac{|a_{k}|}{k^\sigma}\myp,
\end{equation}
where $a_k$'s are the coefficients in the power series expansion of
$H_0(u)$ (see \eqref{eq:H}). Although some of the coefficients $a_k$
may be negative, it turns out that the quantity
$A(1)=\sum_{k=1}^\infty a_k\mypp k^{-1}$, whenever it is finite,
cannot vanish or take a negative value.
\begin{lemma}\label{lm:A(1)>0}
If\/ $A^+(1)<\infty$ then\/ $0<A(1)<\infty$ and the following
equality holds,
\begin{equation}\label{eq:kappa1}
A(1)=\int_0^1\myn u^{-1}\myn H_0(u)\,\dif{u}.
\end{equation}
In particular, the integral in \eqref{eq:kappa1} is convergent.
\end{lemma}
\begin{proof}
From the assumptions on the coefficients $c_k$'s in the expansion
\eqref{eq:F0}, it is evident that for all $u\in(0,1)$ we have
$\mathcal{F}_0(u)=1+\sum_{k=1}^\infty c_k\myp u^k>1$, and hence
$H_0(u)=\ln\myn(\mathcal{F}_0(u))>0$. Furthermore, substituting the
expansion \eqref{eq:H} for $H_0(u)$ and integrating term by term
(which is permissible for power series inside the interval of
convergence), we get for any $s\in(0,1)$
\begin{equation*}
\int_0^s\myn u^{-1}\myn H_0(u)\,\dif{u}=\sum_{k=1}^\infty a_k
\int_0^s\myn u^{k-1}\mypp\dif{u}=\sum_{k=1}^\infty \frac{a_k\myp
s^k}{k}\myp.
\end{equation*}
Passing here to the limit as $s\uparrow1$ and applying to the
right-hand side Abel's theorem on the boundary value of a power
series (see \cite[\S1.22, pp.\:9--10]{TitchTF}), we obtain the
identity~\eqref{eq:kappa1}.
\end{proof}

The quantity $A(1)$ will play a major role in our argumentation; in
particular, it is involved in a suitable calibration of the ``free''
parameter $z$ in the definition \eqref{Q} of the measure $\QQ_z$
(see Section \ref{sec3.1} below).

\subsection{Cumulants of the part counts}\label{sec2.2}

Let us now turn to the random variables $\nu_\ell$ (i.e., the counts
of parts $\ell\in\NN$ in a partition $\lambda\in\CP$). Under the
probability measure $\QQ_z$ (see~\eqref{Q}), the characteristic
function of $\nu_\ell$ is given by\myp\footnote{\ For notational
simplicity, we suppress the dependence on $z$, which should cause no
confusion.}
\begin{equation}\label{eq:c.f.}
\varphi_{\nu_\ell}(t):=\EE_z(\rme^{\myp\rmi
t\nu_\ell})=\frac{\mathcal{F}_0(z^\ell\myp\rme^{\myp\rmi
t})}{\mathcal{F}_0(z^\ell)}\myp, \qquad t\in\RR\myp.
\end{equation}
Hence, the (principal branch of the) logarithm of
$\varphi_{\nu_\ell}(t)$ is expanded using \eqref{eq:H} as
\begin{equation}\label{eq:ln_c.f.}
\ln\myn (\varphi_{\nu_\ell}(t))=H_0(z^\ell\myp\rme^{\myp\rmi
t})-H_0(z^\ell) = \sum_{k=1}^\infty a_k\myp(\rme^{\myp\rmi
kt}-1)\myp z^{k\ell},\qquad t\in\RR\myp.
\end{equation}

For $q\in\NN$, denote by
$m_q[\myp\nu_\ell\myp]:=\EE_{z}(\nu_\ell^{q})$ the \emph{moments} of
the random variable $\nu_\ell$ about zero, and let
$\varkappa_q[\myp\nu_\ell\myp]$ be the \emph{cumulants}, or
\emph{semi-invariants} of $\nu_\ell$ (see, e.g., \cite[\S\myp3.12,
p.\:69]{KS}), defined by the following formal identity in
indeterminant $t$,
\begin{equation}\label{eq:varkappa}
\ln \EE_z(\rme^{\myp\rmi t\nu_\ell})=\sum_{q=1}^\infty \frac{(\rmi
t)^q}{q!}\myp\varkappa_q[\myp\nu_\ell\myp].
\end{equation}
From \eqref{eq:varkappa} it is easy to see (e.g., by taking the
derivative at $t=0$) that the expected value of $\nu_\ell$ coincides
with its first-order cumulant (see \cite[\S\myp3.14, Eq.\,(3.37),
p.\:71]{KS}),
\begin{equation}\label{eq:m1}
\EE_{z}(\nu_\ell)=m_1[\myp\nu_\ell\myp]=\varkappa_1[\myp\nu_\ell\myp].
\end{equation}
Let us also point out the standard expressions for the first few
\emph{central} moments (including the variance) through the
cumulants (see \cite[\S\myp3.14, Eq.\,(3.38), p.\:72]{KS}),
\begin{align}
\label{eq:m2} \Var(\nu_\ell)=
\EE_z\myn\bigl[(\nu_\ell-m_1[\myp\nu_\ell\myp])^2\bigr]&=
\varkappa_2[\myp\nu_\ell\myp],\\[.15pc]
\label{eq:m3}
\EE_z\myn\bigl[(\nu_\ell-m_1[\myp\nu_\ell\myp])^3\bigr]&=\varkappa_3[\myp\nu_\ell\myp],\\
\label{eq:m4}
\EE_z\myn\bigl[(\nu_\ell-m_1[\myp\nu_\ell\myp])^4\bigr]&=\varkappa_4[\myp\nu_\ell\myp]
+ 3\myp(\varkappa_2[\myp\nu_\ell\myp])^2.
\end{align}

\begin{remark}\label{rm:cumX}
The cumulants $\varkappa_q[X]$ of any random variable $X$ are
defined similarly to \eqref{eq:varkappa}; needless to say, the
formulas analogous to \eqref{eq:m1}\,--\,\eqref{eq:m4} also hold
true in the general case (as long as the corresponding moments
exist).
\end{remark}

The next lemma will be instrumental in our analysis.
\begin{lemma}\label{lm:kappa(q)}
The cumulants $\varkappa_q[\myp\nu_\ell\myp]$ are given by
\begin{equation}\label{eq:cumulants}
\varkappa_q[\myp\nu_\ell\myp]=\sum_{k=1}^\infty k^q a_k\myp
z^{k\ell},\qquad q\in\NN.
\end{equation}
In particular,
\begin{equation}\label{eq:m1[nu]}
m_1[\myp\nu_\ell\myp] =\sum_{k=1}^\infty k\myp a_k z^{k\ell}.
\end{equation}
\end{lemma}

\begin{proof}
Taylor expanding the exponential function in \eqref{eq:ln_c.f.}, we
get
\begin{equation}\label{eq:ln(g)}
\ln\myn(\varphi_{\nu_\ell}(t))=\sum_{k=1}^\infty a_k z^{k\ell}
\sum_{q=1}^\infty \frac{(\rmi kt)^q}{q!}
 =\sum_{q=1}^\infty
\frac{(\rmi t)^q}{q!}\sum_{k=1}^\infty k^q a_k\myp z^{k\ell},
\end{equation}
where the interchange of the order of summation in the double series
\eqref{eq:ln(g)} is justified by its absolute convergence. Now, by a
comparison of the expansion \eqref{eq:ln(g)} with the identity
\eqref{eq:varkappa}, the formulas \eqref{eq:cumulants} for the
coefficients $\varkappa_q[\myp\nu_\ell\myp]$ readily follow.
\end{proof}

By virtue of the expression \eqref{eq:m1[nu]} for the expected value
of $\nu_\ell$\myp, it is easy to obtain a formula for the
expectation of $N_\lambda=\sum_{\ell=1}^\infty
\ell\myp\nu_\ell$\myp,
\begin{equation}\label{eq:m1[N]}
\EE_z(N_\lambda)=\sum_{\ell=1}^\infty \ell\mypp
m_1[\myp\nu_\ell\myp]=\sum_{\ell=1}^\infty \ell \sum_{k=1}^\infty
k\myp a_k z^{k\ell}.
\end{equation}
More generally, the expressions \eqref{eq:cumulants} for the
cumulants $\varkappa_q[\myp\nu_\ell\myp]$ furnish a representation
of the cumulants of $N_\lambda$ of any order; namely, using the
rescaling relation $\varkappa_q[\myp\ell
\nu_\ell\myp]=\ell^q\varkappa_q[\myp\nu_\ell\myp]$ (see
\cite[\S\myp3.13, p.\:70]{KS}) and the additivity property of the
cumulants for independent summands (see \cite[\S\myp7.18,
pp.\:\allowbreak201--202]{KS}), we obtain
\begin{equation}\label{eq:cumulants-xi}
\varkappa_q[N_\lambda]=\sum_{\ell=1}^\infty \ell^{\myp q}
\varkappa_q[\myp\nu_\ell\myp]=\sum_{\ell=1}^\infty \ell^{\myp q}
\sum_{k=1}^\infty k^q a_k\myp z^{k\ell},\qquad q\in\NN.
\end{equation}

Similarly, recalling that the upper boundary $Y_\lambda(x)$ of the
Young diagram $\Upsilon_{\myn\lambda}$ is given by the formula
\eqref{eq:Young}, we obtain for any $x\ge0$
\begin{equation}\label{varkappa_zxi*}
\varkappa_q[Y_\lambda(x)]=\sum_{\ell\ge x}
\varkappa_q[\myp\nu_\ell\myp]=\sum_{\ell\ge x} \sum_{k=1}^\infty k^q
a_{k}\myp z^{k\ell},\qquad q\in\NN,
\end{equation}
and in particular (with $q=1$)
\begin{equation}\label{eq:E(t)}
\EE_{z}[Y_\lambda(x)]=\sum_{\ell\ge x} m_1[\myp\nu_\ell\myp] =
\sum_{\ell\ge x}\sum_{k=1}^{\infty} k a_{k} z^{k\ell}.
\end{equation}

\subsection{Estimates for power-exponential sums}\label{sec4.1}

In what follows, we frequently encounter power-exponential sums of
the form
\begin{equation}\label{eq:q}
S_q(t):=\sum_{\ell=1}^\infty \ell^{\myp q-1}\rme^{-t \ell},\qquad
t>0.
\end{equation}

\begin{lemma}\label{lm:6.2}
For\/ $q\in\NN$, the function\/ $S_q(t)$ admits the representation
\begin{equation}\label{eq:S}
S_q(t)=\sum_{j=1}^{q} c_{j,\myp q}\,\frac{\rme^{-t
j}}{(1-\rme^{-t})^j}\myp,\qquad t>0,
\end{equation}
with some constants\/ $c_{j,\myp q}>0$
\,\textup{(}$j=1,\dots,q$\textup{)}\myp\textup{;} in particular,
$c_{q,\myp q}=(q-1)!$\mypp.
\end{lemma}

\begin{proof} If $q=1$ then the expression \eqref{eq:q} is reduced to
a geometric series
\begin{align*}
S_1(t)=\sum_{\ell=1}^\infty \rme^{-t
\ell}=\frac{\rme^{-t}}{1-\rme^{-t}}\myp,
\end{align*}
which is a particular case of \eqref{eq:S} with $c_{1,1}=1$. Assume
now that \eqref{eq:S} is valid for some $q\ge1$. Then,
differentiating the identities \eqref{eq:q} and \eqref{eq:S} with
respect to $t$, we obtain
\begin{align*}
S_{q+1}(t)=-S'_q(t)&=\sum_{j=1}^{q} c_{j,\myp q}\left(\frac{j\mypp
\rme^{-t j}}{(1-\rme^{-t})^j}+ \frac{j\mypp\rme^{-t
(j+1)}}{(1-\rme^{-t})^{j+1}}\right)\\
&=\sum_{j=1}^{q+1} c_{j,\myp q+1}\frac{\rme^{-t
j}}{(1-\rme^{-t})^j}\myp,
\end{align*}
where we set
\begin{equation*}
  c_{j,\myp{}q+1}:=\left\{\begin{array}{ll}
  c_{1,\myp{}q}\myp,&j=1,\\[.2pc]
  j\myp c_{j,\myp{}q}+(j-1)\mypp c_{j-1,\myp{}q}\myp,\ \ &2\le j\le q,\\[.2pc]
  q\mypp c_{q,\myp{}q}\myp,& j=q+1.
  \end{array}
  \right.
\end{equation*}
In particular, $c_{q+1,\myp q+1}=q\mypp
c_{q,\myp{}q}=q\myp(q-1)!=q!$\myp. Thus, the formula \eqref{eq:S}
holds for $q+1$ and hence, by induction, for all $q\ge1$.
\end{proof}

\begin{lemma}\label{lm:exp<C}
For any\/ $q>0$, there is a constant\/ $C_q>0$ such that
\begin{equation}\label{eq:exp_bound}
\frac{\rme^{-t}}{(1-\rme^{-t})^q} \le C_q\mypp t^{-q},\qquad t>0.
\end{equation}
\end{lemma}
\begin{proof}
Set $f(t):=t^q\mypp\rme^{-t}(1-\rme^{-t})^{-q}$ and note that
\begin{equation*}
\lim_{t\to0^+}f(t)=1,\qquad \lim_{t\to +\infty} f(t)=0.
\end{equation*}
By continuity, the function $f(t)$ is bounded on $(0,\infty)$, and
the inequality (\ref{eq:exp_bound}) follows.
\end{proof}

\begin{lemma}\label{lm:C_k}
\textup{(a)} \,For any $q\in\NN$, there is a constant
$\tilde{C}_q>0$ such that
\begin{equation}\label{eq:Cc}
S_q(t)\le \tilde{C}_q \mypp t^{-q},\qquad t>0.
\end{equation}

\textup{(b)} \,Moreover{},
\begin{equation}\label{eq:Sq}
S_q(t)\sim \frac{(q-1)!}{t^q}\myp,\qquad t\to 0^+.
\end{equation}
\end{lemma}

\begin{proof} (a) \,Observe, using Lemma \ref{lm:exp<C}, that for $j=1,\dots,q$
\begin{equation*}
\frac{\rme^{-tj}}{(1-\rme^{-t})^j}\le\frac{\rme^{-t}}{(1-\rme^{-t})^q}\le
C_q\mypp t^{-q},\qquad t>0.
\end{equation*}
Substituting this inequality into \eqref{eq:S} and recalling that
the coefficients $c_{j,\myp q}$ are positive, we obtain the bound
\eqref{eq:Cc} with $\tilde{C}_q:= C_q\sum_{j=1}^q c_{j,\myp q}>0$.

\smallskip
(b) \,For each term in the expansion \eqref{eq:S} we have $\rme^{-t
j}\myp(1-\rme^{-t})^{-j}\sim t^{-j}$ as $t\to0^+$. Hence, the
overall asymptotic behavior of $S_q(t)$ is determined by the term
with $j=q$ and the corresponding coefficient $c_{q,\myp q}=(q-1)!$
(see Lemma~\ref{lm:6.2}), and the formula \eqref{eq:Sq} follows.
\end{proof}

\section{Asymptotics of the expectation}\label{sec3}

\subsection{Calibration of the parameter $z$}\label{sec3.1}
Our aim is to find a suitable parameter $z=z_n\in(0,1)$ in the
definition \eqref{Q1} of the probability measure $\QQ_z$, subject to
the asymptotic condition
\begin{equation}\label{calibr1}
\EE_{z}(N_\lambda)\sim n,\qquad n\to\infty,
\end{equation}
where $N_\lambda=\sum_{\ell=1}^\infty \ell\myp\nu_\ell$\myp. To this
end, let us seek $z$ in the form
\begin{equation}\label{alpha}
z=\rme^{-\alpha},\qquad \alpha=\alpha_n:=\gamma\mypp n^{-1/2},
\end{equation}
where the constant $\gamma>0$ is to be fitted. Hence, the formula
\eqref{eq:m1[N]} takes the form
\begin{equation}\label{eq:E(N)}
\EE_{z}(N_\lambda) =\sum_{\ell=1}^\infty \ell \sum_{k=1}^\infty k
a_{k}\myp \rme^{-k\alpha\ell}.
\end{equation}

Let us state our main result in this section.
\begin{theorem}\label{th:kappa}
Suppose that $A^+(1)<\infty$. Then, under the parameterization
\eqref{alpha}, the asymptotic condition\/ \eqref{calibr1} is
satisfied with the choice
\begin{equation}\label{eq:kappa}
\gamma=\sqrt{A(1)}>0.
\end{equation}
\end{theorem}

\begin{proof}
By Lemma \ref{lm:A(1)>0}, we know that $A(1)>0$ and hence the
inequality \eqref{eq:kappa} holds true.

Let us now investigate the asymptotics of the expectation
$\EE_z(N_\lambda)$ under the parameterization $z=\rme^{-\alpha}$
with $\alpha\to0^+$ (cf.~\eqref{alpha}). Interchanging the order of
summation in \eqref{eq:E(N)} and using the notation \eqref{eq:q}, we
obtain
\begin{equation}\label{E_i'}
\EE_z(N_\lambda)=\sum_{k=1}^\infty k a_{k} \sum_{\ell=1}^\infty
\ell\mypp \rme^{-k\alpha\ell}= \sum_{k=1}^\infty k a_{k}\myp
S_2(k\alpha).
\end{equation}
According to Lemma \ref{lm:C_k}\myp(b) (with $q=2$),\footnote{\ This
can also be seen directly, without Lemma~\ref{lm:C_k}, from the
explicit expression $S_2(t)=\rme^{-t}(1-\rme^{-t})^{-2}$.} for each
$k\in\NN$ we have $S_2(k\alpha)\sim (k\alpha)^{-2}$ as
$\alpha\to0^+$. Moreover, by Lemma \ref{lm:C_k}\myp(a) the general
summand in the series \eqref{E_i'} is bounded, uniformly in $k$, by
$O(\alpha^{-2})|a_k|\myp k^{-1}$, which is a term of a convergent
series since $A^+(1)<\infty$ by the theorem's hypothesis; in
particular, this justifies the above interchange of the order of
summation. Hence, by Lebesgue's dominated convergence theorem we
obtain from \eqref{E_i'}
\begin{equation}\label{zeta32}
\lim_{\alpha\to0^+}\alpha^{2} \EE_{z}(N_\lambda) =\sum_{k=1}^\infty
k a_{k}\myp \lim_{\alpha\to0^+} \alpha^2
S_2(k\alpha)=\sum_{k=1}^\infty \frac{a_{k}}{k}=A(1).
\end{equation}
Thus, putting $\alpha=\gamma\mypp n^{-1/2}$ with
$\gamma=\sqrt{A(1)}$ \,(see~\eqref{eq:kappa}), the limit
\eqref{zeta32} is reduced to \eqref{calibr1}.
\end{proof}

The expression \eqref{eq:kappa1} for $A(1)$ directly in terms of the
generating function $H_0(u)$ is sometimes useful (e.g., for computer
calculations of the coefficient $\gamma=\sqrt{A(1)}$\mypp, see
Example \ref{ex:6} in Section~\ref{sec6}; \,cf.\ also the
formulation of Theorem~\ref{th:main1} in the Introduction).

\begin{assumption}\label{as:z}
Throughout the rest of the paper, we assume that $A^+(1)<\infty$ and
the parameter $z$ is chosen according to the formulas \eqref{alpha}
with $\gamma>0$ defined by \eqref{eq:kappa}.
\end{assumption}

\begin{remark} Under Assumption \ref{as:z} the measure $\QQ_z$ becomes
dependent on $n$, as well as the $\QQ_z$-probabilities and the
corresponding expected values.
\end{remark}

\subsection{The ``expected'' limit shape}\label{sec3.2}

\begin{theorem}\label{th:3.2}
For any $\delta>0$, we have uniformly in\/ $x\in[\delta,\infty)$
\begin{equation}\label{sh}
\EE_{z}\myn\bigl[Y_\lambda(x\myp n^{1/2})\bigr]=n^{1/2}\mypp
\omega^*(x) + O(1),\qquad n\to\infty,
\end{equation}
where the limit shape function $\omega^*(x)$ is defined
in~\eqref{eq:omega}.
\end{theorem}

\begin{proof}
Setting $\ell^*\mynn=\ell_n^*:=\lceil x\myp n^{1/2}\rceil$, in view
of \eqref{alpha} we have
\begin{equation}\label{eq:0<alpha*ell<}
0\le \alpha\myp\ell^*\mynn-\gamma x< \alpha,\qquad n\in\NN,
\end{equation}
and hence, uniformly in $x$,
\begin{equation}\label{eq:alpha*ell}
\alpha\myp\ell^*\mynn\to \gamma x, \qquad n\to\infty.
\end{equation}
With this notation, from \eqref{eq:E(t)} we have for $x>0$
\begin{align}\label{Z1}
\gamma\mypp n^{-1/2}\EE_{z}\myn\bigl[Y_\lambda(x\myp n^{1/2})\bigr]&
=\alpha\sum_{\ell=\ell^*}^\infty \sum_{k=1}^\infty k a_{k}\myp
\rme^{-k\alpha\ell} = \alpha\sum_{\ell=\ell^*}^\infty
g_0(\alpha\ell),
\end{align}
where (see \eqref{eq:H'})
\begin{equation}\label{eq:phi}
g_0(t):=\sum_{k=1}^\infty k a_{k}\myp \rme^{-k t}=\rme^{-t}
H'_0(\rme^{-t}), \qquad t>0.
\end{equation}
Note that (cf.~\eqref{eq:H''})
\begin{equation}\label{eq:phi'}
g_0^{\myp\prime}(t)=-\sum_{k=1}^\infty k^2 a_{k}\myp \rme^{-k
t}=-\left\{\rme^{-2t} H''_0(\rme^{-t})+\rme^{-t}
H'_0(\rme^{-t})\right\},\qquad t>0.
\end{equation}

The right-hand side of \eqref{Z1} can be viewed as a Riemann
integral sum for $g_0(t)$ (over $[\hspace{.03pc}\gamma x,\infty)$
with mesh size $\alpha\to0^+$), suggesting its convergence to the
corresponding integral as $n\to\infty$. More precisely, noting that
$g_0(t)$ is continuously differentiable on $(0,\infty$) and
$g_0(\infty)=0$, by Euler--Maclaurin's summation formula (see, e.g.,
\cite[\S12.2]{Cramer}) we get
\begin{equation}\label{EM0}
\sum_{\ell=\ell^*}^\infty g_0(\alpha\myp\ell)=
\int_{\ell^*}^\infty\! g_0(\alpha\myp t)\,\dif{t} + \tfrac12\mypp
g_0(\alpha\myp\ell^*)+\alpha\myn\int_{\ell^*}^\infty \!
B_1(t)\,g_0^{\myp\prime}(\alpha\myp t)\,\dif{t},
\end{equation}
where $B_1(t):=t-\lfloor t\rfloor -\frac12$ \,($t\in\RR$).
Furthermore, observing that the term $\tfrac12\mypp
g_0(\alpha\myp\ell^*)$ can be included in the last integral yields a
shorter form of \eqref{EM0},
\begin{equation}\label{EM1}
\sum_{\ell=\ell^*}^\infty g_0(\alpha\myp\ell)=
\int_{\ell^*}^\infty\! g_0(\alpha\myp t)\,\dif{t} +
\alpha\myn\int_{\ell^*}^\infty
\mynn\tilde{B}_1(t)\,g_0^{\myp\prime}(\alpha\myp t)\,\dif{t},
\end{equation}
with $\tilde{B}_1(t):=B_1(t)-\frac12\equiv t-\lfloor t\rfloor -1$
\,($t\in\RR$).

\smallskip
For the first integral in \eqref{EM1}, on substituting
\eqref{eq:phi} and using \eqref{eq:alpha*ell} we obtain
\begin{equation}\label{eq:H2}
\int_{\ell^*}^\infty\mynn g_0(\alpha\myp t)\,\dif{t}
=\int_{\ell^*}^\infty \mynn\rme^{-\alpha t} H'_0(\rme^{-\alpha
t})\,\dif{t}
 = \alpha^{-1} H_0(\rme^{-\alpha\ell^*}\myn)\sim \alpha^{-1}
H_0(\rme^{-\gamma x})
\end{equation}
uniformly in $x\ge\delta$, since by Lagrange's mean value theorem,
on account of \eqref{eq:0<alpha*ell<} and~\eqref{eq:phi},
$$
|H_0(\rme^{-\alpha\ell^*}\myn)-H_0(\rme^{-\gamma x})|\le
(\alpha\myp\ell^*\mynn-\gamma x)
\myp\max_{t\ge\gamma\delta}\myn|g_0(t)| =O(\alpha).
$$
Next, noting that $\sup_{t\in\RR}\myn|\tilde{B}_1(t)|\le1$,
recalling that $\alpha\myp\ell^*\mynn\ge \gamma x$ (see
\eqref{eq:0<alpha*ell<}) and substituting \eqref{eq:phi}, the last
term in \eqref{EM1} is bounded in absolute value, again uniformly in
$x\ge \delta$, by
\begin{equation}\label{eq:H1}
\alpha\myn\int_{\ell^*}^\infty\myn |g_0^{\myp\prime}(\alpha\myp
t)|\,\dif{t}\le \int_0^{\myp\rme^{-\gamma\delta}}\!|u
H''_0(u)+H'_0(u)|\,\dif{u}=O(1),
\end{equation}
where we used the expression \eqref{eq:phi'} and the change of
variables $u=\rme^{-\alpha t}$. Thus, substituting the estimates
\eqref{eq:H2}, \eqref{eq:H1} into \eqref{EM1} and returning to
\eqref{Z1}, we obtain
$$
\lim_{n\to\infty} n^{-1/2}\EE_{z}\myn\bigl[Y_\lambda(x\myp
n^{1/2})\bigr]=\gamma^{-1} H_0(\rme^{-\gamma x})\equiv\omega^*(x),
$$
where the convergence is uniform in $x\ge\delta$, as claimed.
\end{proof}

\begin{remark}\label{rm:H(1)} As was mentioned in
Remark~\ref{rm:1.3}, the asymptotic formula \eqref{sh} may be
extended, with obvious adjustments of the proof, to the case $x=0$
including the uniform convergence in $x\ge0$ --- \emph{provided
that} \myp$\omega^*(0)<\infty$ (more precisely, if the function
$H_0(u)$ and its first two derivatives are finite at $u=1$).
\end{remark}

\subsection{Refined asymptotics of the expectation of $N_\lambda$}\label{sec3.3}

We need to sharpen the asymptotics $\EE_z(N_\lambda)-n=o(n)$
provided by Theorem \textup{\ref{th:kappa}} (see~\eqref{calibr1}).
The aim of this section is to prove the following refinement.
\begin{theorem}\label{th:4.1}
Under the condition $A^+(\sigma)<\infty$ with some\/
$\sigma\in(0,1)$, we have
$$
\EE_{z}(N_\lambda)-n=O\bigl(n^{(\sigma+1)/2}\bigr),\qquad
n\to\infty.
$$
\end{theorem}

\subsubsection{Preliminaries.}
For the proof of Theorem \ref{th:4.1}, some preparations are
required. Let $\psi(x)$ be a continuous function on $\RR_+$ such
that $\limsup_{x\to\infty} x^{\beta}\psi(x)<\infty$ with some
$\beta>1$, which ensures that $\psi(x)$ is integrable on $\RR_+$. It
is easy to see that the series
\begin{equation}\label{F0}
\varPsi(h):=\sum_{\ell=1}^\infty \psi(\ell h),\qquad h>0,
\end{equation}
is absolutely convergent, and moreover
\begin{equation}\label{beta}
\varPsi(h)=O(1) \sum_{\ell=1}^\infty (\ell
h)^{-\beta}=O(h^{-\beta}),\qquad h\to\infty.
\end{equation}
Let us also assume that
\begin{equation}\label{eq:h->0}
\varPsi(h)=O(h^{-1}),\qquad h\to 0^+.
\end{equation}

\begin{remark}\label{rm:direct_Riemann} Note that $h\mypp \varPsi(h)
=h\sum_{\ell=1}^\infty \psi(\ell h)$ is a Riemann integral sum for
the function $\psi(x)$ over $\RR_+$ with mesh size $h$, so typically
(including a specific example emerging in the proof of Theorem
\ref{th:4.1}) it will converge, as $h\to0^+$, to the (finite)
integral $\int_0^\infty\myn \psi(x)\,\dif{}x$,\footnote{\ Functions
$\psi(x)$ satisfying this property are called \emph{directly Riemann
integrable} (see \cite[Ch.\,XI, \S1, p.\:362]{Feller2}).} thus
automatically ensuring the bound \eqref{eq:h->0}. A sufficient
condition for such a convergence, which can be verified by using
Euler--Maclaurin's summation formula similar to \eqref{EM1}, is that
$\psi(x)$ be continuously differentiable and the derivative
$\psi'(x)$ absolutely integrable on~$\RR_+$\myp.
\end{remark}

Let us now consider the \emph{Mellin transform} of $\varPsi(h)$
(see, e.g., \cite[Ch.\,VI, \S\myp9]{Widder})
\begin{equation}\label{Mel}
\widehat{\varPsi}(s):=\int_0^\infty \!
h^{s-1}\mypp\varPsi(h)\,\dif{}h\myp, \qquad 1<\Re\myp(s)<\beta.
\end{equation}
On substituting \eqref{F0} into \eqref{Mel} we find
\begin{align}
\notag \widehat{\varPsi}(s)=\int_0^\infty\!
h^{s-1}\sum_{\ell=1}^\infty \psi(\ell h)\,\dif{}h
&=\sum_{\ell=1}^\infty \int_0^\infty\! h^{s-1}\myp \psi(\ell
h)\:\dif{}h
\\
&=\sum_{\ell=1}^\infty \ell^{-s}\int_0^\infty\!
x^{s-1}\myp\psi(x)\,\dif{}x \label{Mel1}
=\zeta(s)\myn\int_0^\infty\! x^{s-1}\myp\psi(x)\,\dif{x}\myp,
\end{align}
where $\zeta(s)\myn=\mynn\sum_{\ell=1}^{\vphantom{y}\infty}\myn{}
\ell^{-s}$ is the Riemann zeta function. The interchange of
summation and integration in this computation is justified by the
absolute convergence of the integral on the right-hand side
of~\eqref{Mel1}. By the well-known properties of $\zeta(s)$ (see,
e.g., \cite[\S\myp2.1, p.\,13]{TitchZF}), from \eqref{beta} and
\eqref{Mel} it follows that the function $\widehat{\varPsi}(s)$ is
meromorphic in the strip \,$0<\Re\myp(s)<\beta$, with a single pole
at\/ $s=1$. Set
\begin{equation}\label{Delta}
\Delta_{\psi}(h):=\varPsi(h)-\frac{1}{h}\int_{0}^\infty\myn
\psi(x)\,\dif{}x, \qquad h>0.
\end{equation}
Then the M\"untz lemma (see \cite[\S\myp2.11,
pp.\:\allowbreak28--29]{TitchZF}) gives
\begin{equation*}
\widehat{\varPsi}(s)=\int_0^\infty\myn
h^{s-1}\Delta_{\psi}(h)\,\dif{}h\myp,\qquad 0<\Re\myp(s)<1,
\end{equation*}
and the inversion formula for the Mellin transform (see, e.g.,
\cite[Ch.\,VI, \S\myp9, Theorem~9a,
pp.\:\allowbreak246--247]{Widder}) implies
\begin{equation}\label{eq:inverse}
\Delta_{\psi}(h)=\frac{1}{2\pi\myp\rmi}\int_{c-\rmi\infty}^{c+\rmi\infty}
h^{-s} \,\widehat{\varPsi}(s)\,\dif{s},\qquad 0<c<1.
\end{equation}

\subsubsection{Proof\/ of\/ Theorem\/ \textup{\ref{th:4.1}}.}\ With the representation
\eqref{eq:inverse} at hand, set $\psi(x):=x\myp\rme^{- \alpha x}$,
then the series \eqref{F0} is explicitly given by
\begin{equation}\label{eq:F(h)}
\varPsi(h)=h\sum_{\ell=1}^\infty \ell\mypp\rme^{- \alpha
h\ell}=\frac{h\mypp \rme^{-\alpha h}}{(1-\rme^{-\alpha
h})^2}\myp,\qquad h>0.
\end{equation}
Clearly, $\psi(x)=O(x^{-\beta})$ with any $\beta>0$, and from
\eqref{eq:F(h)} it is evident that $\varPsi(h)$ satisfies the
condition \eqref{beta} (with $\alpha$ fixed). Furthermore, the
formula \eqref{Mel1} for $\widehat{\varPsi}(s)$ is specialized to
\begin{equation}\label{eq:M_alpha}
\widehat{\varPsi}(s)=\zeta(s)\mynn\int_0^\infty\mynn
x^s\mypp\rme^{-\alpha x}\,\dif{x}=\alpha^{-s-1}
\zeta(s)\mypp\Gamma(s+1),\qquad 1<\Re\myp(s)<\infty,
\end{equation}
where $\Gamma(s)=\int_0^\infty\mynn x^{s-1} \rme^{-x}\,\dif{x}$ is
the gamma function. Since $\Gamma(s+1)$ is analytic for
$\Re\myp(s)>-1$ (cf.\ \cite[\S\myp4.41, p.\,148]{TitchTF}) and, as
already mentioned, $\zeta(s)$ has a single (simple) pole at point
$s=1$, it follows that the expression \eqref{eq:M_alpha} is
meromorphic in the half-plane $\Re\myp(s)>-1$, thus providing an
analytic continuation of the function $\widehat{\varPsi}(s)$ into
the strip $-1<\Re\myp(s)<1$.

Combining \eqref{E_i'} and \eqref{eq:F(h)} we get
\begin{equation}\label{E_z''}
\EE_{z}(N_\lambda)=\sum_{k=1}^\infty a_{k}\myp\varPsi(k).
\end{equation}
On the other hand, according to the notation \eqref{eq:A2-} and
Assumption~\ref{as:z} we have the identity
\begin{equation}\label{eq:int(f)}
\sum_{k=1}^\infty \frac{a_{k}}{k\myp\alpha^2}= \frac{n\myp
A(1)}{\gamma^2} \equiv n.
\end{equation}
Consequently, subtracting \eqref{eq:int(f)} from \eqref{E_z''} we
obtain the representation
\begin{equation}\label{dif1}
\EE_{z}(N_\lambda)-n=\sum_{k=1}^\infty
a_{k}\left(\varPsi(k)-\frac{1}{k\alpha^{2}}\right)=\sum_{k=1}^\infty
a_{k}\mypp\Delta_{\psi}\mynn(k),
\end{equation}
recalling the notation \eqref{Delta} and observing that
\begin{equation*}
\int_{0}^\infty\! \psi(x)\,\dif{}x= \int_0^\infty\mynn
x\mypp\rme^{-\alpha x}\,\dif{}x= \frac{1}{\alpha^{2}}\myp.
\end{equation*}

Furthermore, using the representation \eqref{eq:inverse} with
$c=\sigma\in(0,1)$ (see the hypothesis of the theorem) and
substituting the expression \eqref{eq:M_alpha}, we can rewrite
\eqref{dif1} in the form
\begin{align}
\notag \EE_{z}(N_\lambda)-n&=\frac{1}{2\pi\myp \rmi}
\sum_{k=1}^\infty a_{k} \int_{\sigma-\rmi\infty}^{\sigma+\rmi\infty}
\frac{\zeta(s)\mypp\Gamma(s+1)}{\alpha^{s+1}k^{s}}\,\dif{s}\\[.2pc]
\label{dif4} &=\frac{1}{2\pi}\int_{-\infty}^{\infty}
\frac{A(\sigma+\rmi t)\mypp\zeta(\sigma+\rmi
t)\mypp\Gamma(\sigma+1+\rmi t)}{\alpha^{\sigma+1+\rmi t}}\,\dif{}t,
\end{align}
using the change of variables $s=\sigma+\rmi\myp t$. To justify the
interchange of summation and integration deployed in \eqref{dif4},
note that
\begin{equation*}
|A(\sigma+\rmi t)|\le A^+(\sigma)<\infty,\qquad
|\alpha^{-\sigma-1-\rmi t}| \le \alpha^{-\sigma-1}\myp.
\end{equation*}
We can also use the following classical estimates as $t\to\infty$
(see \cite[Theorem~1.9, p.\:25]{Iv} and \cite[\S\myp4.42,
p.\,151]{TitchTF}, respectively),
\begin{equation*}
\zeta(\sigma+\rmi
t)=O\myp\bigl(|t|^{(1-\sigma)/2}\ln\myn(|t|+2)\bigr),\qquad
\Gamma(\sigma+1+\rmi t)=O\myp\bigl(|t|^{\sigma+1/2}\mypp
\rme^{-\pi|t|/2}\bigr).
\end{equation*}
Hence, the last integral in \eqref{dif4} is bounded, uniformly in
$n\in\NN$, by
\begin{equation}\label{eq:int<infty}
O(\alpha^{-\sigma-1})\int_{-\infty}^\infty\myn |t|^{1+\sigma/2}\mypp
\rme^{-\pi|t|/2}
\ln\myn(|t|+2))\,\dif{}t=O(\alpha^{-\sigma-1})<\infty,
\end{equation}
which validates the formula \eqref{dif4}.

Moreover, combining \eqref{dif4} and \eqref{eq:int<infty} we get, on
account of~\eqref{alpha},
\begin{align*}
\EE_{z}(N_\lambda)-n &=
O(\alpha^{-\sigma-1})=O\bigl(n^{(\sigma+1)/2}\bigr),\qquad
n\to\infty,
\end{align*}
and the proof of Theorem \ref{th:4.1} is complete.

\section{Asymptotic estimates for higher-order moments}\label{sec4}

\subsection{The cumulants of $N_\lambda$}\label{sec4.2}

Substituting $z=\rme^{-\alpha}$ (see~\eqref{alpha}) into the
formulas \eqref{eq:cumulants-xi} for the cumulants of $N_\lambda$,
we get
\begin{equation}\label{eq:cumulants-xi**}
\varkappa_q[N_\lambda]=\sum_{\ell=1}^\infty \ell^{\myp q}
\sum_{k=1}^\infty k^q a_k\myp\rme^{-k\alpha\ell},\qquad q\in\NN.
\end{equation}
Recall that Assumption \ref{as:z} is presumed to be satisfied
throughout.

\begin{theorem}\label{th:kappa_q}
For each\/ $q\in\NN$,
\begin{equation}\label{eq:kappa_q}
\varkappa_q[N_\lambda]\sim
\frac{q!}{\gamma^{q-1}}\,n^{(q+1)/2},\qquad n\to\infty.
\end{equation}
In particular, the variance of $N_\lambda$ satisfies
\begin{equation}\label{eq:Sigma}
\Var(N_\lambda)\sim \frac{2}{\gamma}\, n^{3/2},\qquad n\to\infty.
\end{equation}
\end{theorem}

\begin{proof}
The proof follows the same lines as that of Theorem \ref{th:kappa}
(i.e., with $q=1$). Namely, again using Lemma \ref{lm:C_k} and
Lebesgue's dominated convergence theorem, from
\eqref{eq:cumulants-xi**} we get
\begin{align}\label{eq:q!kappa}
\alpha^{q+1}\varkappa_q[N_\lambda] =\sum_{k=1}^\infty k^q a_{k}\myp
\bigl(\alpha^{q+1}S_{q+1}(k\alpha)\bigr) \to q!\sum_{k=1}^\infty
\frac{a_{k}}{k}=q!\,A(1)\equiv q!\,\gamma^2.
\end{align}
But $\alpha^{q+1}\mynn\sim \gamma^{q+1} n^{-(q+1)/2}$
(see~\eqref{alpha}), and the limit \eqref{eq:q!kappa} is reduced
to~\eqref{eq:kappa_q}.

The second claim of the theorem (i.e., the asymptotic
formula~\eqref{eq:Sigma}) immediately follows from
\eqref{eq:kappa_q} with $q=2$ by noting that
$\Var(N_\lambda)=\varkappa_2[N_\lambda]$ \mypp(cf.~\eqref{eq:m2}).
\end{proof}

\subsection{The cumulants of\/ $Y_\lambda(x)$}\label{sec4.3}

With the substitution $z=\rme^{-\alpha}$, the representations
\eqref{varkappa_zxi*} are rewritten in the form
\begin{equation}\label{varkappa_zxi**}
\varkappa_q[Y_\lambda(x)]=\sum_{\ell\ge x} \sum_{k=1}^\infty k^q
a_{k}\mypp \rme^{-k\alpha\ell},\qquad q\in\NN.
\end{equation}

Let us first consider the case $q=2$, where
$\varkappa_2[Y_\lambda(x)]=\Var[Y_\lambda(x)]$ (see~\eqref{eq:m2}).
\begin{theorem}\label{th:K1}
For every\/ $x>0$,
\begin{equation}\label{eq:Var(Y)}
\lim_{n\to\infty}n^{-1/2}\mypp\Var\myn\bigl[Y_\lambda(x\myp
n^{1/2})\bigr]= \gamma^{-1} \rme^{-\gamma x} H'_0(\rme^{-\gamma x}),
\end{equation}
where the convergence is uniform in $x\in[\delta,\infty)$ for any\/
$\delta>0$.
\end{theorem}

\begin{proof}
With the help of the notation $g_0(t)$ defined in \eqref{eq:phi}
(see also \eqref{eq:phi'}), formula \eqref{varkappa_zxi**} (with
$q=2$) takes the form
\begin{equation}\label{D1_1(t)*}
\Var\myn\bigl[Y_\lambda(x\myp n^{1/2})\bigr]=\sum_{\ell\ge x
n^{1/2}}  \sum_{k=1}^\infty k^2 a_{k}\mypp \rme^{-k\alpha\ell}
=-\myn\sum_{\ell\ge x n^{1/2}}\myn g_0^{\myp\prime}(\alpha\myp
\ell)\myp.
\end{equation}
Interpreting the right-hand side of \eqref{D1_1(t)*} as a Riemann
integral sum and arguing as in the proof of Theorem~\ref{th:3.2}, we
deduce that the equation \eqref{D1_1(t)*} converges, uniformly in
$x\ge\delta$, to
\begin{align*}
\lim_{n\to\infty}n^{-1/2}\mypp\Var\myn\bigl[Y_\lambda(x\myp
n^{1/2})\bigr]&=-\myp\gamma^{-1}\!\int_{\gamma x}^\infty
g_0^{\myp\prime}(t)\,\dif{t}\\[.2pc]
&=\gamma^{-1} g_0(\gamma x) =\gamma^{-1} \rme^{-\gamma x}
H'_0(\rme^{-\gamma x}),
\end{align*}
according to \eqref{eq:phi}. Thus, the theorem is proved.
\end{proof}

It is straightforward to adapt the proof of Theorem \ref{th:K1} to
the case $q\ge3$, which only requires a standard generalization of
the differential formulas \eqref{eq:H'}, \eqref{eq:H''} to higher
orders. This way, one can obtain the asymptotics of the form
$$
\lim_{n\to\infty}n^{-1/2}\myp\varkappa_q\myn\bigl[Y_\lambda(x\myp
n^{1/2})\bigr]= \chi_\gamma(x),\qquad x>0,
$$
where the function $\chi_\gamma(x)$ is expressed in terms of the
derivatives $H_0^{(j)}(\rme^{-\gamma x})$ ($j=1,\dots,q$).

For the purposes of the present paper (more precisely, for the proof
of Lemma \ref{lm:Kq4} below), we only need an \emph{upper estimate}
as follows.

\begin{lemma}\label{lm:Kq}
For every\/ $q\in\NN$ and any\/ $\delta>0$ we have, uniformly in
$x\in[\delta,\infty)$,
\begin{equation}\label{eq:rho1}
\varkappa_{q}\bigl[Y_\lambda(x\myp n^{1/2})\bigr]=O(n^{1/2}),\qquad
n\to\infty.
\end{equation}
\end{lemma}

In Section \ref{sec5.4} we will require the asymptotics (in fact, an
asymptotic bound) for the fourth \emph{central moment} of
$Y_\lambda(x\myp n^{1/2})$, which is established next.
\begin{lemma}\label{lm:Kq4}
Set $Y^0_\lambda(t):=Y_\lambda(t)-\EE_{z}[Y_\lambda(t)]$. Then for
any\/ $\delta>0$, uniformly in $x\in[\delta,\infty)$,
\begin{equation}\label{eq:rho10}
\lim_{n\to\infty} n^{-1}\EE_z\myn\bigl[\bigl(Y_\lambda^0(x\myp
n^{1/2})\bigr)^4\bigr]= 3\mypp\bigl\{\gamma^{-1} \rme^{-\gamma x}
H'_0(\rme^{-\gamma x})\bigr\}^2.
\end{equation}
\end{lemma}
\begin{proof}
Using the formula \eqref{eq:m4} (which is valid for any random
variable) we have
\begin{align*}
\EE_z\myn\bigl[\bigl(Y_\lambda^0(x\myp
n^{1/2})\bigr)^4\bigr]&=\varkappa_4\bigl[Y_\lambda(x\myp
n^{1/2})\bigr]+3\mypp\bigl\{\varkappa_2\bigl[Y_\lambda(x\myp
n^{1/2})\bigr]\bigr\}^2\\
&=O(n^{1/2})+3\myp n\left\{\gamma^{-1} \rme^{-\gamma x}
H'_0(\rme^{-\gamma x})\right\}^2\left(1+o(1)\right),\qquad
n\to\infty,
\end{align*}
on account of the (uniform) estimates \eqref{eq:Var(Y)} and
\eqref{eq:rho1}. Hence, the limit \eqref{eq:rho10} follows.
\end{proof}

\begin{remark}\label{rm:H(1)'}
Similarly to Remark \ref{rm:H(1)}, all the results above are valid
also for $x=0$ provided that the function $H_0(u)$ and the
corresponding derivatives are finite at $u=1$.
\end{remark}

\subsection{The Lyapunov ratio}\label{sec:Lyapunov}
Let us introduce the \emph{Lyapunov ratio} (of the third order)
\begin{equation}\label{L3}
L_z:=\frac{1}{\sigma_{\myn z}^{3}} \sum_{\ell=1}^\infty\ell^{\myp 3}
\mu_3[\myp\nu_\ell\myp],
\end{equation}
where we denote for short $\sigma_{\myn
z}:=\sqrt{\myp\Var(N_\lambda)}$ and
\begin{equation*}
\mu_3[\myp\nu_\ell\myp]
:=\EE_{z}\myn\bigl[|\nu_\ell^0|^3\bigr],\qquad
\nu_\ell^0:=\nu_\ell-m_1[\myp\nu_\ell\myp]
\end{equation*}
(i.e., $\mu_3[\myp\nu_\ell\myp]$ is the third-order absolute central
moment of $\nu_\ell$). The next asymptotic estimate will play an
important role in the proof of the local limit theorem in Section
\ref{sec5.3} below.

\begin{lemma}\label{lm:7.1} Suppose that $A^+(\frac12)<\infty$.
Then
\begin{equation}\label{eq:L}
L_z\asymp n^{-1/4},\qquad n\to\infty.
\end{equation}
\end{lemma}

\begin{proof}
In view of the definition \eqref{L3} and the asymptotics
$\sigma_{\myn z}\asymp n^{3/4}$ provided by Theorem~\ref{th:kappa_q}
(see~\eqref{eq:Sigma}), for the proof of \eqref{eq:L} it suffices to
show that
\begin{equation}\label{eq:sum(mu3)}
M_3:=\sum_{\ell=1}^\infty\ell^{\myp 3} \mu_3[\myp\nu_\ell\myp]\asymp
n^{2},\qquad n\to\infty.
\end{equation}

Starting with a \emph{lower bound} for $M_3$, observe using the
relation \eqref{eq:m3} that
\begin{equation}\label{eq:mu3lower}
\mu_3[\myp\nu_\ell\myp] \ge
m_3[\myp\nu_\ell^0\myp]=\varkappa_3[\myp\nu_\ell\myp].
\end{equation}
Hence, on account of the formula~\eqref{eq:cumulants-xi} and Theorem
\ref{th:kappa_q} (with $q=3$), from \eqref{eq:sum(mu3)} we get
\begin{equation}\label{eq:M-lower}
M_3\ge \sum_{\ell=1}^\infty \ell^{\myp 3}
\varkappa_3[\myp\nu_\ell\myp]=\varkappa_3[N_\lambda]\asymp
n^{2},\qquad n\to\infty,
\end{equation}
which is in agreement with the claim~\eqref{eq:sum(mu3)}.

To obtain a suitable \emph{upper bound} on $M_3$, note that for any
$u,v\ge0$,
\begin{equation}\label{eq:|u-v|}
|u-v|^3=(u-v)^2|u-v|\le (u-v)^2(u+v)=(u-v)^3 +2\myp v\myp (u-v)^2.
\end{equation}
Setting $u=\nu_\ell$\mypp, \,$v=m_1[\myp\nu_\ell\myp]$ in
\eqref{eq:|u-v|} and taking the expectation, we get the inequality
\begin{equation}\label{eq:mu3upper}
\mu_3[\myp\nu_\ell\myp] \le m_3[\myp\nu_\ell^0\myp] +2\myp
m_1[\myp\nu_\ell\myp]\cdot m_2[\myp\nu_\ell^0\myp]
=\varkappa_3[\myp\nu_\ell\myp]+2\varkappa_1[\myp\nu_\ell\myp]\cdot\varkappa_2[\myp\nu_\ell\myp],
\end{equation}
according to the identities \eqref{eq:m1}\,--\,\eqref{eq:m3}. Note
that the term $\varkappa_3[\myp\nu_\ell\myp]$ here is the same as in
\eqref{eq:mu3lower}, and so gives the contribution of the order of
$n^2$ into the corresponding upper bound for $M_3$, which is
consistent with the lower bound \eqref{eq:M-lower}.

The remaining product term on the right-hand side of
\eqref{eq:mu3upper}, when elaborated using \eqref{eq:cumulants}
(with $q=1$ and $q=2$, respectively) and substituted into
\eqref{eq:sum(mu3)}, yields
\begin{align}
\notag \sum_{\ell=1}^\infty
\ell^3\varkappa_1[\myp\nu_\ell\myp]\,\varkappa_2[\myp\nu_\ell\myp]&=\sum_{\ell=1}^\infty
\ell^3\sum_{k=1}^\infty ka_k\myp\rme^{-k\alpha\ell}\sum_{m=1}^\infty
m^2a_m\myp\rme^{-m\alpha\ell}\\
\notag &=\mynn \sum_{k,\myp m\ge1} k\myp |a_k| \:m^2
|a_m|\,S_4\bigl((k+m)\myp\alpha\bigr)\\
\label{eq:(k+m)^4} &=O(\alpha^{-4}) \sum_{k,\myp m\ge 1} \frac{k\myp
|a_k| \: m^2 |a_m|}{(k+m)^4}\myp,
\end{align}
according to Lemma~\ref{lm:C_k}. Observing that for $k,m\ge1$
\begin{equation*}
(k+m)^4=(k+m)^{3/2}(k+m)^{5/2}\ge  k^{3/2}\myp m^{5/2},
\end{equation*}
the right-hand side of \eqref{eq:(k+m)^4} is further estimated by
\begin{equation*}
O(\alpha^{-4}) \sum_{k=1}^\infty \frac{|a_k|}{k^{1/2}}
\sum_{m=1}^\infty \frac{|a_m|}{m^{1/2}}= O(\alpha^{-4}) \mypp
\bigl(A^+(\tfrac12)\bigr)^2=O(n^2),
\end{equation*}
according to the lemma's hypothesis and the asymptotics
$\alpha\asymp n^{-1/2}$ (see~\eqref{alpha}).

Thus, we have shown that $M_3=O(n^2)$, and together with the lower
bound \eqref{eq:M-lower} this completes the proof
of~\eqref{eq:sum(mu3)}.
\end{proof}

\section{A local limit theorem and the limit shape}\label{sec5}

\subsection{Statement of the local limit theorem}\label{sec5.1}
The role of a local limit theorem in our approach is to yield the
asymptotics of the probability
$\QQ_z\{N_\lambda=n\}\equiv\QQ_z(\CP_n)$ appearing in the
representation of the measure $\PP_n$ as a conditional distribution,
$\PP_n(\cdot)=\QQ_z(\cdot\mypp|\CP_n) =\QQ_z(\cdot)/\QQ_z(\CP_n)$.

To prove such a theorem (see Theorem~\ref{th:LCLT} below), we
will require a technical condition on the generating function
$\mathcal{F}_0(u)$ as follows.
\begin{assumption}\label{as:7.1}
There exists a constant $\delta_*\myn>0$ such that for any
$\theta\in(0,1)$ the function $H_0(u)=\ln\myn(\mathcal{F}_0(u))$
\myp($u\in\CC$) satisfies the inequality
\begin{equation}\label{eq:<a1}
H_0(\theta) -\Re\myp(H_0(\theta\mypp\rme^{\myp\rmi t}))\ge \delta_*
\mypp\theta\mypp(1-\cos t),\qquad t\in\RR\myp.
\end{equation}
\end{assumption}

\begin{remark}\label{rm:a1>0}
In terms of the coefficients $\{a_k\}$ in the expansion
\eqref{eq:F0}, the left-hand side of \eqref{eq:<a1} is expressed as
$\sum_{k=1}^\infty\myn a_k\myp \theta^k(1-\cos kt)$. Consequently,
if $a_1>0$ and $a_k\ge0$ for all $k\ge2$ then the inequality
\eqref{eq:<a1} is satisfied with $\delta_*=a_1>0$.
\end{remark}

As before, we denote $\mu_z=\EE_{z}(N_\lambda)$, \,$\sigma_{\myn z}=
\sqrt{\myp\Var(N_\lambda)}$\myp. Consider the probability density of
a normal distribution $\mathcal{N}(\mu_z,\sigma_{\myn z}^2)$ (i.e.,
with mean $\mu_z$ and variance $\sigma_{\myn z}^2$),
\begin{equation}\label{eq:phi1}
f_{\mu_z\myn,\myp\sigma_{\myn z}}(x)=
\frac{1}{\sqrt{2\pi}\,\sigma_{\myn
z}}\,\exp\!\myn\left\{-{\textstyle \frac12}\myp
(x-\mu_z)^2\myn/\sigma_{\myn z}^{2}\mypp\right\},\qquad x\in\RR\myp.
\end{equation}

\begin{theorem}\label{th:LCLT}
Let $A^+(\frac12)<\infty$ and Assumption\/ \textup{\ref{as:7.1}}
hold. Then, uniformly in\/ $m\in\ZZ_+$\myp,
\begin{equation}\label{eq:LCLT}
\QQ_{z}\{N_\lambda=m\}=f_{\mu_z\myn,\myp\sigma_{\myn
z}}(m)+O(n^{-1}),\qquad n\to\infty.
\end{equation}
\end{theorem}

In fact we will only need a particular case with $m=n$.

\begin{corollary}\label{cor:Q} Under the conditions of\/
Theorem \textup{\ref{th:LCLT}},
\begin{equation}\label{sim}
\QQ_{z}\{N_\lambda=n\}\asymp n^{-3/4}, \qquad n\to\infty.
\end{equation}
\end{corollary}

With the asymptotic results of Sections \ref{sec3.3} and
\ref{sec4.3} at hand, it is not difficult to deduce the corollary
from the theorem.

\begin{proof}[Proof\/ of\/ Corollary\/ \textup{\ref{cor:Q}}] By
Theorem~\ref{th:4.1} with $\sigma=\frac12$\myp, we have $\mu_z=
n+O(n^{3/4})$. Together with Theorem~\ref{th:kappa_q}
(see~\eqref{eq:Sigma}) this implies that $(n-\mu_z)/\sigma_{\myn z}
=O(1)$. Hence,
\begin{align*}
f_{\mu_z\myn,\myp\sigma_{\myn
z}}(n)&=\frac{1}{\sqrt{2\pi}\,\sigma_{\myn z}}\,
\exp\mynn\bigl\{-{\textstyle\frac12}\myp(n-\mu_z)^2\myn/\sigma_{\myn
z}^{2}\bigr\}\asymp \sigma_{\myn z}^{-1}\sim n^{-3/4}, \qquad
n\to\infty,
\end{align*}
and \eqref{sim} now readily follows from~\eqref{eq:LCLT}.
\end{proof}

\subsection{Estimates of the characteristic functions}\label{sec5.2}

For the proof of Theorem \ref{th:LCLT}, we need some technical
preparations. Recall from Section \ref{sec2.1} that the random
variables $\{\nu_{\ell}\myp,\,\ell\in\NN\}$ are mutually independent
under the measure $\QQ_{z}$. Hence, the characteristic function
$\varphi_{N_\lambda}(t)=\EE_{z} (\rme^{\myp\rmi tN_\lambda})$ of the
sum $N_\lambda=\sum_{\ell=1}^\infty \myn\ell\myp\nu_\ell$ \myp{}is
given by
\begin{equation}\label{x.f_5_0}
\varphi_{N_\lambda}(t) =\prod_{\ell=1}^\infty \varphi_{\nu_\ell}(t
\ell)=\prod_{\ell=1}^\infty
\frac{\mathcal{F}_0(z^{\ell}\rme^{\myp\rmi
t\ell})}{\mathcal{F}_0(z^{\ell})}\myp,\qquad t\in\RR\myp,
\end{equation}
where $\varphi_{\nu_\ell}(\cdot)$ is the characteristic function of
$\nu_\ell$ (see~\eqref{eq:c.f.}). The next lemma provides a useful
estimate for $\varphi_{N_\lambda}(t)$ essentially proved in
\cite[Lemma 7.12]{BZ4}.\footnote{\ A ``two-dimensional'' proof in
\cite[Lemma 7.12]{BZ4} can be easily adapted to the one-dimensional
case.} Recall that the Lyapunov ratio $L_z$ is defined
in~\eqref{L3}.
\begin{lemma}\label{lm:7.2_F}
For all\/ $t\in\RR$ such that $|t|\le (L_z\sigma_z)^{-1}$ we have
\begin{equation*}
\bigl|\varphi_{N_\lambda}(t)-\exp\mynn\bigl\{\rmi t
\mu_z-\tfrac12\mypp t^2\sigma_{\myn z}^2\bigr\} \mynn\bigr| \le
16\myp|t|^3 L_z\myp\sigma_{\myn z}^3 \exp\mynn\bigl\{-\tfrac16\mypp
t^2\sigma_{\myn z}^2\bigr\}.
\end{equation*}
\end{lemma}

Let us also prove the following global bound.
\begin{lemma}\label{lm:7.3}
Suppose that Assumption\/ \textup{\ref{as:7.1}} is satisfied
\textup{(}with $\delta_*>0$\textup{)}. Then
\begin{equation}\label{f_J}
|\varphi_{N_\lambda}(t)|\le\exp\mynn\bigl\{-\delta_*J_\alpha(t)\bigr\},\qquad
t\in\RR\myp,
\end{equation}
where
\begin{equation}\label{J_0}
J_{\alpha}(t):=\sum_{\ell=1}^\infty \rme^{-\alpha\ell}\myp(1-\cos
t\ell\myp).
\end{equation}
\end{lemma}

\begin{proof}
From \eqref{x.f_5_0} it follows that
\begin{equation}\label{eq:J_1}
\ln\mynn |\varphi_{N_\lambda}(t)| = \sum_{\ell=1}^\infty
\ln\mynn|\varphi_{\nu_\ell}(t\ell)|, \qquad t\in\RR\myp.
\end{equation}
Furthermore, using \eqref{eq:ln_c.f.} and Assumption \ref{as:7.1}
with $\theta=z^\ell$ (see \eqref{eq:<a1}), for each $\ell\in\NN$ we
have
\begin{align}
\notag
\ln\myn|\varphi_{\nu_\ell}(t\ell)|
=\Re\myp\bigl(\ln\myn(\varphi_{\nu_\ell}(t\ell))\bigr) &=
\Re\myp(H_0(z^\ell\myp\rme^{\myp\rmi t\ell})) -H_0(z^\ell)\\[.2pc]
\label{eq:<-delta*}
&\le -\delta_* \myp z^{\ell}\myp(1-\cos
t\ell\myp),\qquad t\in\RR\myp.
\end{align}
Setting here $z=\rme^{-\alpha}$ (see~\eqref{alpha}) and returning
from \eqref{eq:<-delta*} to \eqref{eq:J_1}, we obtain the inequality
$\ln\mynn |\varphi_{N_\lambda}(t)|\le -\delta_* J_\alpha(t)$, which
is equivalent to~\eqref{f_J}.
\end{proof}

\subsection{Proof of Theorem \ref{th:LCLT}}
\label{sec5.3}

By definition, the characteristic function $\varphi_{N_\lambda}(t)
=\EE_z\myn\bigl[\rme^{\myp \rmi t N_\lambda}\bigr]$ is given by
\begin{equation}\label{eq:Fourier}
\varphi_{N_\lambda}(t) =\sum_{m=0}^\infty
\QQ_{z}\{N_\lambda=m\}\,\rme^{\myp\rmi t\myp m},\qquad t\in\RR\myp.
\end{equation}
Hence, the coefficients of the Fourier series \eqref{eq:Fourier} are
expressed as
\begin{equation}\label{l_1}
\QQ_{z}\{N_\lambda=m\}=\frac{1}{2\pi}\myn\int_{T} \rme^{-\rmi t\myp
m}\mypp \varphi_{N_\lambda}(t)\,\dif{}t,\qquad m\in\ZZ_{+}\myp,
\end{equation}
where $T:=[-\pi,\pi]$. On the other hand, the characteristic
function of the normal distribution $\mathcal{N}(\mu_z,\sigma_{\myn
z}^2)$ (see~\eqref{eq:phi1}) is given by
\begin{equation*}
\int_{-\infty}^\infty\myn f_{\mu_z\myn,\myp\sigma_{\myn
z}}\myn(x)\,\rme^{\myp\rmi t x}\,\dif{x}= \rme^{\myp\rmi t
\mu_z-t^2\sigma_{\myn z}^2/2},\qquad t\in\RR\myp,
\end{equation*}
so by the inversion formula we have
\begin{equation}\label{f_o}
f_{\mu_z\myn,\myp\sigma_{\myn z}}\myn(m)= \frac{1}{2\pi}
\myn\int_{-\infty}^\infty \rme^{-\rmi t \myp m}\,\rme^{\myp\rmi t
\myp \mu_z-t^2\sigma_{\myn z}^2/2}\,\dif{}t\myp,\qquad
m\in\ZZ_+\myp.
\end{equation}

Denote $D_z\myn:=\{t\in\RR\myp\colon |t|>(L_z\sigma_{\myn
z})^{-1}\}$. By the asymptotic formula \eqref{eq:Sigma} and
Lemma~\ref{lm:7.1}, we have $(L_z\sigma_{\myn z})^{-1}\asymp
n^{1/4}\myp n^{-3/4}=n^{-1/2}=o(1)$, which implies that
$D_z^c:=\RR\setminus D_z\subset T$ for all $n$ large enough.
Furthermore, since $\alpha\asymp n^{-1/2}$ (see~\eqref{alpha}), it
follows that $(L_z\sigma_{\myn z})^{-1}\!>\eta\myp\alpha$ with a
suitable (small) constant $\eta>0$, hence $D_z\subset \{t\in
\RR\colon |t|> \eta\myp\alpha\}$. Thus, subtracting \eqref{f_o}
from~\eqref{l_1} we get, uniformly in $m\in\ZZ_+$\myp,
\begin{equation}\label{I}
\bigl|\QQ_{z}\{N_\lambda=m\}-f_{\mu_z\myn,\myp\sigma_{\myn
z}}\myn(m)\bigr|\le \mathcal{I}_1+\mathcal{I}_2+\mathcal{I}_3\myp,
\end{equation}
where
\begin{gather}\label{eq:I12}
\mathcal{I}_1:=\frac{1}{2\pi}\myn\int_{D_z^c}
      \bigl|\varphi_{N_\lambda}(t)-\rme^{\myp\rmi t \myp\mu_z
                              -t^2\sigma_{\myn z}^2/2}
      \bigr|\,\dif{}t,\qquad \mathcal{I}_2:=\frac{1}{2\pi}\myn\int_{D_z}
      \!\rme^{-t^2\sigma_{\myn z}^2/2}\,\dif{}t,\\
\label{eq:I3} \mathcal{I}_3:=\frac{1}{2\pi}\myn
      \int_{T\cap D_z}
      \!|\varphi_{N_\lambda}(t)|\:\dif{}t.
\end{gather}
By Lemma \ref{lm:7.2_F} and on the substitution
$t=y\mypp\sigma_{\myn z}^{-1}$, the integral $\mathcal{I}_1$ in
\eqref{eq:I12} is estimated by
\begin{equation}\label{I1}
\mathcal{I}_1 =O(L_z\myp\sigma_{\myn z}^{-1})\mynn\int_{0}^\infty
\mynn y^3\mypp\rme^{-y^2\myn/6}\,\dif{}y= O(n^{-1}),
\end{equation}
according to the asymptotics of $\sigma_{\myn z}$ and $L_z$ (see
\eqref{eq:Sigma} and \eqref{eq:L}, respectively). Similarly, for the
integral $\mathcal{I}_2$ (see \eqref{eq:I12}) we obtain, again using
\eqref{eq:Sigma} and \eqref{eq:L},
\begin{align}
\notag \mathcal{I}_2= O(\sigma_{\myn z}^{-1})
\mynn\int_{L_z^{-1}}^\infty \myn\rme^{-y^2\myn/2}\,\dif{}y&=O(L_z
\sigma_{\myn z}^{-1})\mynn\int_{L_z^{-1}}^\infty
\myn y\,\rme^{-y^2\myn/2}\,\dif{}y\\[.2pc]
\label{I2} &= O(n^{-1/2})\,
\rme^{-L_z^{-2}\myn/2}=O\bigl(n^{-1/2}\myp\rme^{-\const
\sqrt{n}}\myp\bigr)=o(n^{-1}).
\end{align}

Finally, let us turn to the integral $\mathcal{I}_3$ in
\eqref{eq:I3}. By Lemma \ref{lm:7.3} and a remark about the domain
$D_z$ made before display \eqref{I}, we have
\begin{align}\label{I3}
\mathcal{I}_3&\le \frac{1}{2\pi}\mynn \int_{T\cap D_z}
\!\rme^{-\delta_*J_{\alpha}(t)}\,\dif{t}\le
\frac{1}{\pi}\mynn\int_{\eta\alpha}^\pi
\rme^{-\delta_*J_{\alpha}(t)}\,\dif{t}.
\end{align}
Furthermore, evaluating the sum in \eqref{J_0} (where for
convenience we include the vanishing term with $\ell=0$) we obtain
\begin{align}
\notag J_{\alpha}(t)= \sum_{\ell=0}^\infty \rme^{-\alpha
\ell}\myp\!\left(1-\Re\myp(\rme^{\myp\rmi t
\ell}\myp)\right)&=\frac{1}{1-\rme^{-\alpha}}- \Re\myp\!\left(\frac{
1}{1-\rme^{-\alpha+\rmi t}}\right)\\
\label{J_1}& \ge
\frac{1}{1-\rme^{-\alpha}}-\frac{1}{|1-\rme^{-\alpha + \rmi
t}|}\myp,
\end{align}
because $\Re\myp(s)\le|s|$ for any $s\in\CC$\myp. Observe that for
$t\in[\eta\myp\alpha,\pi]$
\begin{equation*}
|1-\rme^{-\alpha+\rmi t}|
\ge|1-\rme^{-\alpha+\rmi\eta\alpha}|\sim\alpha \mypp|1+\rmi \eta|
=\alpha \mypp\sqrt{1+\eta^2} \qquad (\alpha\to0^+).
\end{equation*}
Substituting this estimate into \eqref{J_1}, we conclude that
$J_\alpha(t)$ is asymptotically bounded below by
$C(\eta)\myp\alpha^{-1}\myn \asymp n^{1/2}$ (with
$C(\eta)=1-(1+\eta^2)^{-1/2}>0$), uniformly in
$t\in[\eta\myp\alpha,\pi]$. Thus, the integral in \eqref{I3} is
bounded by $O\bigl(\rme^{-\const \, \sqrt{n}}\myp\bigr)=o(n^{-1})$.

Hence, recalling also the estimates \eqref{I1} and \eqref{I2}, we
see that the right-hand side of \eqref{I} admits an asymptotic bound
$O(n^{-1})$, which completes the proof of Theorem~\ref{th:LCLT}.

\subsection{The limit shape results}\label{sec5.4}
Recall the definition $\omega^*(x):=\gamma^{-1}H_0(\rme^{-\gamma
x})$ (see~\eqref{eq:omega}), where
$H_0(u)=\ln\myn(\mathcal{F}_0(u))$ and $\gamma=\bigl(\int_0^1\myn
u^{-1}H_0(u)\,\dif{u}\bigr)^{1/2}$ (see \eqref{eq:kappa}
and~\eqref{eq:kappa1}).
\begin{theorem}\label{th:LSQ}
Under Assumption \textup{\ref{as:z}} we have, for every\/ $\delta>0$
and any\/ $\varepsilon>0$,
\begin{equation*}
\lim_{n\to\infty} \QQ_{z}\bigl\{\lambda\in\CP\colon
\sup\nolimits_{x\ge\delta}\mynn|\myp n^{-1/2}\myp Y_\lambda(x\myp
n^{1/2})-\omega^*(x)|>\varepsilon\bigr\}=1.
\end{equation*}
\end{theorem}

\begin{proof}
By virtue of Theorem \ref{th:3.2}, letting
$Y^0_\lambda(t):=Y_\lambda(t)-\EE_{z}[Y_\lambda(t)]$ it suffices to
check that
\begin{equation}\label{eq:Q->0}
\lim_{n\to\infty}\QQ_{z}\bigl\{ \sup\nolimits_{x\ge\delta}
\mynn|Y^{0}_\lambda(x\myp n^{1/2})|>\varepsilon\myp
n^{1/2}\bigr\}\to0.
\end{equation}
Put $Z_\lambda(t):=Y_\lambda(t^{-1})$ ($t>0$). From the definition
\eqref{eq:Young} of $Y_\lambda(\cdot)$, for any $0<s<t$ we have
$$
Z_\lambda(t)-Z_\lambda(s)=Y_\lambda(t^{-1})-Y_\lambda(s^{-1})=\sum_{t^{-1}\le
\myp\ell\myp<s^{-1}}\!\nu_\ell\mypp,
$$
and it follows that the random process $Z_\lambda(t)$ ($t>0$) has
independent increments. Hence,
$Z^0_\lambda(t):=Z_\lambda(t)-\EE_{z}[Z_\lambda(t)]$ \strut{}is a
martingale with respect to the filtration ${\mathcal
F}_t=\sigma\{\nu_\ell\myp,\, \ell\ge t^{-1}\}$. From
\eqref{eq:Young} it is also evident that $Z^0_\lambda(t)$ is
\emph{c\`adl\`ag} (i.e., its paths are everywhere right-continuous
and have left limits, cf.\ Fig.\mypp\ref{fig1}a). Therefore, by the
Doob--Kolmogorov submartingale inequality (see, e.g.,
\cite[Theorem~6.14, p.\:99]{Yeh}
we obtain
\begin{align}
\notag \QQ_{z}\bigl\{ \sup\nolimits_{x\ge\delta}
\mynn|Y^{0}_\lambda(x\myp n^{1/2})|>\varepsilon\myp n^{1/2}\bigr\}
&\equiv \QQ_{z}\bigl\{ \sup\nolimits_{y\le\delta^{-1}}
\mynn|Z^{0}_\lambda(y\myp
n^{-1/2})|>\varepsilon\myp n^{1/2}\bigr\}\\[.2pc]
\notag
&\le\frac{\sup_{y\le\delta^{-1}}\mynn\Var\myn\bigl[Z_\lambda(y\myp
n^{-1/2})\bigr]}{\varepsilon^2 n}\\[.2pc]
\notag
&\le \frac{\Var\myn\bigl[Z_\lambda(\delta^{-1}
n^{-1/2})\bigr]}{\varepsilon^2
n}\\[.2pc]
\label{eq:Doob1} &\equiv \frac{\Var\myn\bigl[Y_\lambda(\delta\myp
n^{1/2})\bigr]}{\varepsilon^2 n}=O(n^{-1/2}),
\end{align}
in view of Theorem~\ref{th:K1}. Thus, the claim \eqref{eq:Q->0}
follows and the theorem is proved.
\end{proof}

We are finally ready to prove our main result about the limit shape
under the measure $\PP_n$ (cf.\ Theorem \ref{th:main1} stated in the
Introduction).

\begin{theorem}\label{th:LSP}
Suppose that $A^+(\frac12)<\infty$ and that Assumption
\textup{\ref{as:7.1}} is satisfied. Then, for every\/ $\delta>0$ and
any\/ $\varepsilon>0$,
\begin{equation*}
\lim_{n\to\infty} \PP_n\bigl\{\lambda\in\CP_n\colon
\sup\nolimits_{x\ge \delta}\mynn\bigl|\myp n^{-1/2}\myp
Y_\lambda(x\myp n^{1/2})-\omega^*(x)\bigr| >\varepsilon\bigr\}=0.
\end{equation*}
\end{theorem}

\begin{proof}
Like in the proof of Theorem \ref{th:LSQ}, the claim is reduced to
the limit
\begin{equation}\label{A1-0}
\lim_{n\to\infty}\PP_n\bigl\{\sup\nolimits_{x\ge\delta}\mynn
|Y_\lambda^{0}(x\myp n^{1/2})|>\varepsilon\myp n^{1/2}\bigr\}=0,
\end{equation}
with $Y^0_\lambda(t)=Y_\lambda(t)-\EE_{z}[Y_\lambda(t)]$. Recalling
the definition \eqref{Pn} of $\PP_n(\cdot)$, it is easy to see that
\begin{equation}\label{A1}
\PP_n\bigl\{\sup\nolimits_{x\ge\delta}\mynn|Y_{\lambda}^0 (x\myp
n^{1/2})|> \varepsilon\myp n^{1/2}\bigr\} \le
\frac{\QQ_{z}\bigl\{\sup_{x\ge \delta}\mynn|Y_\lambda^{0}(x\myp
n^{1/2})|> \varepsilon\myp n^{1/2}\bigr\}}
{\QQ_{z}\{N_\lambda=n\}}\myp.
\end{equation}
Again using the time reversal $t\mapsto t^{-1}$ as in the proof of
Theorem \ref{th:LSQ} and applying the Doob--Kolmogorov submartingale
inequality (now with the fourth moment), we obtain
(cf.~\eqref{eq:Doob1})
\begin{equation*}
\QQ_{z}\bigl\{\sup\nolimits_{x\ge \delta}\mynn|Y^0_\lambda(x\myp
n^{1/2})|>\varepsilon\myp n^{1/2} \bigr\}\le
\frac{\EE_{z}\myn\bigl[\bigl(Y^0_\lambda(\delta\myp
n^{1/2})\bigr)^4\bigr]}{\varepsilon^4\myp n^{2}} =O(n^{-1}),
\end{equation*}
by Lemma \ref{lm:Kq4}. On the other hand, for the denominator in
\eqref{A1} we have $\QQ_{z}\{N_\lambda=n\}\asymp n^{-3/4}$ by
Corollary~\ref{cor:Q}. As a result, the right-hand side of
\eqref{A1} is dominated by $O(n^{-1/4})=o(1)$, and the limit
\eqref{A1-0} readily follows.
\end{proof}

\section{Examples}\label{sec6}

We now proceed to a few illustrative examples of multiplicative
ensembles of random partitions with equiweighted parts. As we will
see, some of the examples entail simple representatives of the three
meta-classes of decomposable combinatorial structures known as
\emph{assemblies}, \emph{multisets} and \emph{selections} (see
\cite[\S\myp2.2]{ABT}). More specifically, Example \ref{ex:1} below
belongs to the class of weighted partitions, including the case of
unrestricted partitions under the uniform (equiprobable)
distribution; Example \ref{ex:2} leads to (weighted) partitions with
bounds on the part counts, including uniformly distributed strict
partitions (i.e., with distinct parts); Example \ref{ex:3} includes
set partitions with labeled elements and ordered contents. Examples
\ref{ex:4} and \ref{ex:5}, as well as Example \ref{ex:3}, are
instances of the so-called \emph{exponential structures} (see, e.g.,
\cite[\S\myp5.5]{Stanley}).
To the best of our knowledge, Example \ref{ex:6} appears to be new
in the context of random partitions; interestingly, it furnishes a
\emph{branch point singularity} of the generating function
$\mathcal{F}_0(u)$ at $u=1$ (see a discussion at the end of
Section~\ref{sec1.2}).

\subsection{Assemblies, multisets, and selections: a synopsis}\label{sec6.1}

A brief account below essentially follows the classic book
\cite{ABT} (see also the earlier paper~\cite{AT}).

A decomposable combinatorial structure defined on $n\in\NN$ elements
is characterized by the (non-ordered) collection of its components
of sizes $\ell=1,2,\dots$ with the corresponding counts
(multiplicities) $\nu_1,\nu_2,\dots$, so that $\sum_{\ell=1}^n
\ell\myp \nu_\ell=n$. Consequently, the counts $\{\nu_\ell\}$
determine a partition $\lambda=(1^{\nu_1}2^{\nu_2}\mynn\dots)$ of
the integer $n$. The specific composition of each component may or
may not be relevant, depending on whether the elements are
distinguishable (``labeled'') or not. Furthermore, suppose that
components of the same size may vary by their type; more
specifically, given a sequence of natural numbers $\{m_\ell\}$,
suppose that a component of size $\ell\in\NN$ may be colored in
$m_\ell$ different colors, irrespectively of any other components.

Let $\mathcal{S}_n=\{s\}$ be the set of all admissible instances $s$
of such a structure of size $n\in\ZZ_+$\myp, and denote their number
by $p(n):=\#\mathcal{S}_n$ (by convention, $\mathcal{S}_0=\emptyset$
and $p(0):=1$). Suppose that the space $\mathcal{S}_n$ is endowed
with a uniform probability measure, whereby all $p(n)$ instances
$s\in\mathcal{S}_n$ are equally likely; in turn, this induces a
certain probability distribution $\PP_n$ on the corresponding random
counts $(\nu_1,\dots,\nu_n)$ and, consequently, on the partition
space $\CP_n$\myp.

This general scheme is exemplified by the three aforementioned
meta-types of decomposable combinatorial structures. In brief,
\emph{assemblies} are formed of labeled exchangeable elements,
whereas in \emph{multisets} the elements are unlabeled and therefore
indistinguishable; furthermore, \emph{selections} are like multisets
but with distinct components. In what follows, we elaborate on that
by giving formulas for the respective generating functions (which in
all cases enjoy a product decomposition of the form~\eqref{eq:F0}),
as well as for the corresponding joint distributions of the random
counts $\nu_\ell$'s under the uniform parent measure on the space
$\mathcal{S}_n$ (which should be compared with the general
multiplicative formula~\eqref{eq:b-Gamma-mult}).

\subsubsection{Assemblies.}\label{sec6.1.1}
This class is characterized by the formula \cite[\S\myp2.2,
p.\,46]{ABT} (cf.\ \cite[\S\myp5.1]{Stanley})
\begin{equation}\label{eq:exp(M)1}
\mathcal{F}(z):=\sum_{n=0}^\infty\frac{p(n)\myp
z^n}{n!}=\exp\!\left(\sum_{\ell=1}^\infty\frac{m_\ell\mypp
z^\ell}{\ell!}\right)\!,
\end{equation}
which fits in the definition \eqref{eq:F0} of multiplicative
measures with the constituent generating functions
$\mathcal{F}_{\myn\ell\myp}(z):=\exp\mynn(m_\ell\mypp z/\ell!)$ and
the corresponding power expansion coefficients
\begin{equation*}
c_k^{(\ell)}=\left(\frac{m_\ell}{\ell!}\right)^k\myn\frac{1}{k!}\myp,\qquad
k\in\ZZ_+\myp.
\end{equation*}

It is easy to show (see \cite[Eq.\,(2.2), p.\,46]{ABT}) that the
number of assemblies of size $n$ which have prescribed counts
$\nu_\ell=k_\ell$\myp, \,$\ell=1,\dots,n$ (satisfying the condition
$\sum_{\ell=1}^n \myn\ell k_\ell=n$) is equal~to
\begin{equation}\label{eq:nu=k}
n!\prod_{\ell=1}^n
\left(\frac{m_\ell}{\ell!}\right)^{k_\ell}\myn\frac{1}{k_\ell\myp!}\mypp,
\end{equation}
and it follows that the joint distribution of $\nu_\ell$'s in this
model is given by \cite[Eq.\,(2.6), p.\,48]{ABT}
$$
\PP_n\{\nu_\ell=k_\ell,\
\ell=1,\dots,n\}=\frac{n!}{p(n)}\prod_{\ell=1}^n
\left(\frac{m_\ell}{\ell!}\right)^{k_\ell}\myn\frac{1}{k_\ell\myp!}\mypp,\qquad
\sum_{\ell=1}^n \ell k_\ell=n\myp.
$$

A simple subclass of assemblies is obtained by setting $m_\ell\equiv
m\in\NN$, which may be interpreted as equiprobable colored set
partitions with labeled elements $\{1,\dots,n\}$; the case $m=1$
thus corresponds to plain set partitions with uniform distribution.

\subsubsection{Multisets.}\label{sec6.1.2}
This class is determined by the generating function \cite[\S\myp2.2,
p.\,47]{ABT}
\begin{equation*}
\mathcal{F}(z):=\sum_{n=0}^\infty p(n)\myp z^n=\prod_{\ell=1}^\infty
(1-z^{\ell})^{-m_\ell},
\end{equation*}
which satisfies the multiplicative definition \eqref{eq:F0} with
\begin{equation*}
\mathcal{F}_{\myn\ell\myp}(z):=(1-z)^{-m_\ell}\mynn=\exp\!\left(m_\ell
\sum_{j=1}^\infty \frac{z^j}{j}\right),\qquad \ell\in\NN,
\end{equation*}
and the corresponding coefficients
\begin{equation*}
c_k^{(\ell)}=\binom{m_\ell+k-1}{k},\qquad k\in\ZZ_+\myp.
\end{equation*}

Here, the joint distribution of $\nu_\ell$'s is given by (see
\cite[Eqs.\ (2.3), (2.9)]{ABT})
$$
\PP_n\{\nu_\ell=k_\ell,\
\ell=1,\dots,n\}=\frac{1}{p(n)}\prod_{\ell=1}^n
\binom{m_\ell+k_\ell-1}{k_\ell}, \qquad \sum_{\ell=1}^n \ell
k_\ell=n\myp.
$$

The particular case $m_\ell\equiv m\in\NN$ corresponds to weighted
integer partitions, which for $m=1$ is reduced to the plain
(unrestricted) partitions with uniform distribution on the
space~$\CP_n$\myp.

\subsubsection{Selections.}\label{sec6.1.3}
This class is defined by the generating function \cite[\S\myp2.2,
p.\,47]{ABT}
\begin{equation*}
\mathcal{F}(z):=\sum_{n=0}^\infty p(n)\myp z^n=\prod_{\ell=1}^\infty
(1+z^{\ell})^{m_\ell}.
\end{equation*}
Hence, $\mathcal{F}(z)$ satisfies the definition \eqref{eq:F0} with
\begin{equation*}
\mathcal{F}_{\myn\ell\myp}(z):=(1+z)^{m_\ell}=\exp\!\left(m_\ell
\sum_{j=1}^\infty
 \frac{(-z)^j}{j}\right)\!,\qquad \ell\in\NN,
\end{equation*}
and the coefficients
\begin{equation*}
c_k^{(\ell)}=\binom{m_\ell}{k},\qquad k=0,1,\dots,m_\ell\myp.
\end{equation*}

The joint distribution of $\nu_\ell$'s is given by (see \cite[Eqs.\
(2.4), (2.12)]{ABT})
$$
\PP_n\{\nu_\ell=k_\ell,\
\ell=1,\dots,n\}=\frac{1}{p(n)}\prod_{\ell=1}^n
\binom{m_\ell}{k_\ell}\myp,\qquad \sum_{\ell=1}^n \ell k_\ell=n\myp.
$$

The case $m_\ell\equiv m\in\NN$ entails integer partitions with part
counts capped by $m$; for $m=1$ this is further reduced to strict
partitions (i.e., with distinct parts) under the uniform
distribution on the corresponding space $\check{\CP}_n$\myp.
\subsection{The generating functions}\label{sec6.2}
In this section, we introduce six examples by specifying the
generating function $\mathcal{F}_0(u)=\sum_{k=0}^\infty c_k\myp u^k$
and the corresponding function
$H_0(u)=\ln\myn(\mathcal{F}_0(u))=\sum_{k=1}^\infty a_k\myp u^k$.
Although the associated multiplicative measures $Q_z$ and $P_n$ are
defined primarily in terms of the coefficients $\{c_k\}$ (see
\eqref{Q1} and \eqref{eq:P_n}, respectively), the explicit
expressions for $c_k$'s may be complicated, so we will not always
attempt to give such expressions.

For our purposes, it is more important to focus on the function
$H_0(u)$ and its power expansion coefficients $\{a_k\}$, since these
are the ingredients that determine the existence and exact form of
the limit shape $\omega^*(x)=\gamma^{-1}H_0(\rme^{-\gamma x})$
(see~\eqref{eq:omega}), including the parameter $\gamma$ (see
\eqref{eq:kappa} and~\eqref{eq:kappa1}). In particular, we have to
check the basic condition $A^+(1)<\infty$ (see
Assumption~\ref{as:z}), as well as the refined condition
$A^+(\frac12)<\infty$ and Assumption \ref{as:7.1}, both needed for
the limit shape result under the measure $\PP_n$ (see
Theorem~\ref{th:LSP}).

\begin{example}\label{ex:1}
For $r\in(0,\infty)$, $\rho\in(0,1]$, set
\begin{equation}\label{eq:F01}
\mathcal{F}_0(u):=(1-\rho\myp u)^{-r}, \qquad |u|<\rho^{-1}.
\end{equation}
By the binomial formula, the coefficients in the power series
expansion \eqref{eq:F0} are given by
\begin{equation}\label{eq:b-k-r}
c_k:=\binom{r+k-1}{k}\myp\rho^k=\frac{r(r+1)\cdots(r+k-1)}{k!}\,
\rho^k\myp,\qquad k\in\ZZ_+\myp.
\end{equation}
In particular, $c_0=1$ and, moreover, $c_k>0$ for all $k\in\NN$.

\begin{remark}
The parameter $\rho<1$ introduces exponential weights of the part
counts, which discourages multiple occurrences of the same part as
compared to the neutral case $\rho=1$. The parameter $r$ also
contributes to the weighting; e.g., if $\rho=1$ then $c_{k+1}/c_k>1$
whenever $r>1$. The combined effect of the parameters $\rho<1$ and
$r>\rho^{-1}>1$ is more interesting: it is easy to see that the
maximum of the sequence $c_k$ is attained for (integer) $k=k^*$ near
$(r\rho-1)/(1-\rho)$.
\end{remark}

For $\rho=1$ and $r=m\in\NN$, the formula \eqref{eq:F01} pinpoints a
multiset structure (see Section~\ref{sec6.1.2}) arising via
partitioning an integer $n\in\NN$ into parts, each of which is then
colored in one of $m$ different colors, irrespectively of its size.
The simplest case $\rho=1$, $r=1$ thus corresponds to the classical
ensemble of uniform integer partitions mentioned in the Introduction
(Sections~\ref{sec1.1},~\ref{sec1.2}).

Note that formula \eqref{Q} for the $\QQ_z$-distribution of the part
counts $\nu_{\myn\ell}$ ($\ell\in\NN$) specializes to
\begin{equation}\label{Qb1}
\QQ_{z}\{\nu_\ell=k\}=\binom{r+k-1}{k}\, \rho^k z^{k\ell}\myp(1-\rho
z^\ell)^r,\qquad k\in\ZZ_+\myp,
\end{equation}
which is a negative binomial distribution with parameters $r$ and
$p=1-\rho z^\ell$ \myp\cite[Ch.\,VI, \S\myp8, p.\,165]{Feller1}. If
$r=1$ then $\mathcal{F}_0(u)=(1-\rho\myp u)^{-1}$, \,$c_k=\rho^k$,
and \eqref{Qb1} is reduced to a geometric distribution
\begin{equation*}
\QQ_{z}\{\nu_\ell=k\}= \rho^k z^{k\ell}\myp(1-\rho z^\ell),\qquad
k\in\ZZ_+\myp.
\end{equation*}
In the latter case, from \eqref{eq:b-Gamma-mult} we get
\begin{equation}\label{eq:r=1}
\PP_n(\lambda)=\mathfrak{C}_n^{-1}\rho^{\myp Y_{\lambda}(0)},\qquad
\lambda\in\CP_n\myp,
\end{equation}
where $Y_\lambda(0)= \sum_{\ell=1}^\infty\myn \nu_\ell$ is the total
number of parts in partition
$\lambda=(1^{\nu_1}2^{\nu_2}\mynn\dots)$
\mypp(cf.~\eqref{eq:Young}). If also $\rho=1$ then \eqref{eq:r=1} is
further reduced to the uniform distribution on $\CP_n$\myp.

Returning to the general case, from \eqref{eq:F01} we have
\begin{equation}\label{eq:H01}
H_0(u)=-r\ln\myn(1-\rho\myp u)=r\sum_{k=1}^\infty
\frac{\,\rho^k}{k}\mypp u^k.
\end{equation}
Since the coefficients in the expansion \eqref{eq:H01} are positive,
Assumption \ref{as:7.1} is satisfied by Remark~\ref{rm:a1>0}; also,
it readily follows that $A^+(\sigma)<\infty$ for any $\sigma>0$ (and
each $\rho\in(0,1]$).
\end{example}

\begin{example}\label{ex:2}
For \,$m\in\NN$, \,$\rho\in(0,1]$, consider the generating function
\begin{equation}\label{eq:F02}
\mathcal{F}_0(u):=(1+\rho\myp u)^m,\qquad u\in\CC,
\end{equation}
with the coefficients
\begin{equation*}
c_k=\binom{m}{k}\myp\rho^k=\frac{m(m-1)\cdots(m-k+1)}{k!}\,\rho^k,\qquad
k=0,1,\dots,m.
\end{equation*}
Consequently, formula \eqref{Q} gives a binomial distribution
\begin{equation}\label{Qb2}
\QQ_{z}\{\nu_\ell=k\}=\binom{m}{k} \frac{\rho^k z^{k
\ell}}{\left(1+\rho z^\ell\right)^{m}}\myp,\qquad k=0,1,\dots,m,
\end{equation}
with parameters $m$ and $p=\rho z^\ell(1+\rho z^\ell)^{-1}$.

Setting $\rho=1$ in \eqref{eq:F02} yields selections (see
Section~\ref{sec6.1.3}) corresponding to integer partitions with
multiplicities $\nu_\ell\le m$ ($\ell\in\NN$); in particular, $m=1$
corresponds to strict partitions (see Sections
\ref{sec1.1},~\ref{sec1.2}). More generally, for $m=1$ and
$0<\rho\le1$, the measure $Q_z$ is concentrated on the subspace
$\check{\CP}\subset\CP$ with the distribution \eqref{Qb2} reduced to
\begin{equation*}
\QQ_{z}\{\nu_\ell=k\}= \frac{\rho^k z^{k\ell}}{1+\rho
z^\ell}\myp,\qquad k=0,1.
\end{equation*}
Accordingly, formula \eqref{eq:b-Gamma-mult} specifies on
$\check{\CP}_n$ the weighted distribution (cf.~\eqref{eq:r=1})
\begin{equation*}
\PP_n(\lambda)=\check{\mathfrak{C}}_n^{-1}\rho^{\myp
Y_\lambda(0)},\qquad \lambda\in \check{\CP}_{n\myp},
\end{equation*}
which is reduced to the uniform distribution if $\rho=1$, as already
mentioned.

In the general case, from \eqref{eq:F02} it follows
\begin{equation}\label{eq:H02}
H_0(u)=m\myp\ln\myn(1+\rho\myp u)=m\sum_{k=1}^\infty
\frac{(-1)^{k-1}\myp \rho^k}{k}\, u^k,
\end{equation}
and it is evident that $A^+(\sigma)<\infty$ for each $\sigma>0$ (and
any $\rho\in(0,1]$). Finally, let us verify Assumption \ref{as:7.1}.
Using \eqref{eq:H02} we obtain, for any $\theta\in(0,1)$ and all
$t\in\RR$\myp,
\begin{align*}
H_0(\theta)-\Re\myp(H_0(\theta\mypp\rme^{\rmi t}))
&=m\myp\ln\myn(1+\rho \mypp\theta)-m\mypp\Re\myp\bigl(\ln\myn(1+\rho
\mypp\theta\mypp\rme^{\rmi t})\bigr)\\
&=m\myp\ln\myn(1+\rho\mypp\theta)-m\myp\ln\myn|1+\rho\mypp\theta\mypp\rme^{\rmi t}|\\
&=-\frac{m}{2}\ln\!\myn\left(\frac{1+2\rho\mypp\theta\cos t+\rho^2\theta^2}{(1+\rho\mypp\theta)^2}\right)\\
&\ge -\frac{m}{2}\left(\frac{1+2\rho\mypp\theta\cos t+
\rho^2\theta^2}{(1+\rho\mypp\theta)^2}-1\right)\\
&= \frac{m\rho\,\theta\mypp(1-\cos t)}{(1+\rho\mypp\theta)^2}\ge
\frac{m\rho}{(1+\rho)^2}\,\theta\mypp(1-\cos t)\myp.
\end{align*}
Thus, the inequality \eqref{eq:<a1} holds with
$\delta_*=m\rho/(1+\rho)^2>0$.
\end{example}

\begin{example}\label{ex:3}
For $b\in(0,\infty)$, \,$\rho\in[0,1]$, consider the generating
function
\begin{equation}\label{eq:F03}
\mathcal{F}_0(u):=\exp\!\mynn\left(\frac{b\myp u}{1-\rho\myp
u}\right), \qquad |u|<\rho^{-1}.
\end{equation}
Noting that $(1-t)^{-1}=\sum_{k=0}^\infty t^k$ (with $t=\rho s$), it
is evident that the coefficients $c_{k}$'s in the power series
expansion of the function \eqref{eq:F03} are positive, with $c_0=1$,
\,$c_1=b$, \,$c_2=b\rho+\frac12\myp b^2$, etc. More systematically,
by Fa\`a di Bruno's formula generalizing the chain rule of
differentiation to higher derivatives (see \cite[Ch.\:I, \S12,
p.\:34]{Jordan}) we obtain
\begin{equation}\label{eq:Faa}
c_k=\sum_{m=1}^k  b^{\myp m}\rho^{\myp
k-m}\!\sum_{(j_1\myn,\dots,\myp j_k)
\myp\in\myp\mathcal{J}_m}\frac{1}{j_1\myn!\cdots j_k !}\mypp, \qquad
k\in\NN,
\end{equation}
where $\mathcal{J}_m$ is the set of all nonnegative integer
$k$-tuples $(j_1,\dots,j_k)$ such that $j_1+\dots+j_k=m$ and
$j_1+2j_2+\dots+k\myp j_k =k$.

\begin{remark}\label{rm:1-to-1}
Note that the $k$-tuples $(j_1,\dots,j_k)\in\mathcal{J}_m$ are in
one-to-one correspondence with partitions of $k$ involving precisely
$m$ different integers as parts, where each element $j_\ell$ has the
meaning of the multiplicity of part $\ell\in\{1,\dots,k\}$.
\end{remark}

\begin{remark}\label{rm:Faa}
For $b=\rho=1$, the formula \eqref{eq:Faa} is reduced, on account of
Remark \ref{rm:1-to-1}, to
$$
c_k=\sum_{m=1}^k \sum_{(j_1\myn,\dots,\myp j_k)
\myp\in\myp\mathcal{J}_m}\frac{1}{j_1\myn!\cdots j_k
!}=\sum_{\lambda\vdash k}\frac{1}{\nu_1\myn!\cdots \nu_k !}\mypp,
\qquad \lambda=(1^{\nu_1} 2^{\nu_2}\mynn\dots)\in\CP_k.
$$
From the formula \eqref{eq:nu=k} with $m_\ell=\ell!$\myp, it follows
that the quantity $p(k)=k!\myp c_k$ equals the number of partitions
of the set $\{1,\dots,k\}$ into components with \emph{ordered}
contents. Such a structure may be visualized as a \emph{forest of
linear rooted trees} (i.e., a disjoint union of connected directed
acyclic graphs, where each vertex has at most two neighbors), with
labeled vertices.
\end{remark}

The observation made in Remark \ref{rm:Faa} can be explained without
calculations using the general theory of assemblies (see
Section~\ref{sec6.1.1}). Namely, the function \eqref{eq:F03} with
$b=\rho=1$ may be represented in the exponential form
\eqref{eq:exp(M)1} by setting $m_\ell:=\ell!$ \myp($\ell\in\NN$),
\begin{equation}\label{eq:assemblyEx3}
\mathcal{F}_0(u)=\exp\!\mynn\left(\frac{u}{1-u}\right)
=\exp\!\mynn\left(\myn\sum_{\ell=1}^\infty
u^\ell\right)=\exp\!\mynn\left(\myn\sum_{\ell=1}^\infty
\frac{m_\ell\, u^\ell}{\ell!}\right)\mynn.
\end{equation}
In the terminology of Section \ref{sec6.1}, that is to say that in
the corresponding assembly each part of size $\ell$ is colored in
one of $m_\ell=\ell!$ different colors, which is equivalent to
ordering the content of this part in one of $\ell!$ ways. Hence, on
comparing the power series expansions \eqref{eq:F0} and
\eqref{eq:exp(M)1} for $\mathcal{F}_0(u)$, we conclude that
$c_k=p(k)/k!$, where $p(k)$ equals the total number of instances of
such an assembly of size~$n$, in accord with Remark~\ref{rm:Faa}.

\smallskip
If $\rho=0$ then the generating function \eqref{eq:F03} is reduced
to $\mathcal{F}_0(u)=\rme^{\myp b\myp u}$, with the expression
\eqref{eq:Faa} simplified to $c_k=b^k/k!$ \myp($k\in\ZZ_+$). Hence,
according to \eqref{Q} the counts $\nu_\ell$ in a random partition
$\lambda=(1^{\nu_1}2^{\nu_2}\mynn\dots)\in\CP$ have a Poisson
distribution with parameter \,$b\myp z^\ell$,
\begin{equation*}
\QQ_{z}\{\nu_\ell=k\}=\frac{b^k z^{k\ell}}{k!}\,\rme^{-b
z^\ell},\qquad k\in\ZZ_+\myp,
\end{equation*}
which leads to the distribution on $\CP_n$ of the form
(see~\eqref{eq:b-Gamma-mult})
\begin{equation*}
\PP_n(\lambda)=\mathfrak{C}_n^{-1}\prod_{\ell=1}^\infty
\frac{b^{\myp\nu_\ell}}{\nu_\ell!}\myp,\qquad \lambda\in\CP_n\myp.
\end{equation*}
The parameter $b>0$ here determines an exponential weighting: having
more parts of each size is either encouraged or discouraged
according as $b>1$ or $b<1$.

In the special case $\rho=0$, $b\in\NN$, the multiplicative ensemble
defined via the product formula \eqref{eq:gfF} admits a simple
combinatorial interpretation. Indeed, similarly to
\eqref{eq:assemblyEx3} the exponential identity \eqref{eq:exp(M)1}
with $m_\ell:=b\mypp \ell!$ \myp($\ell\in\NN$) determines an
assembly of size $n\in\NN$ obtained by partitioning the set
$\{1,\dots,n\}$ into non-empty blocks with ordered contents, each
block colored in one of $b$ different colors irrespectively of its
size.

\smallskip
Returning to the general formula \eqref{eq:F03}, we get
\begin{equation}\label{eq:H03}
H_0(u)=\frac{b\myp u}{1-\rho\myp u}=b\sum_{k=1}^\infty \,\rho^{k-1}
u^{k},
\end{equation}
and hence Assumption \ref{as:7.1} is automatic (see
Remark~\ref{rm:a1>0}); moreover, $A^+(\sigma)<\infty$ for any
$\sigma>0$, except for $\rho=1$ whereby $A^+(\sigma)<\infty$ only
with $\sigma>1$.
\end{example}

\begin{example}\label{ex:4}
Extending Example \ref{ex:3} (for simplicity, with $b=1$), let us
set for $r>0$, $r\ne1$ and $\rho\in(0,1]$
\begin{equation}\label{eq:beta4a}
\mathcal{F}_0(u):=\exp\!\mynn\left(\frac{u}{(1- \rho\myp
u)^r}\right)\!,\qquad |u|<\rho^{-1}.
\end{equation}
Taking the logarithm of \eqref{eq:beta4a} we get the power series
expansion (cf.~\eqref{eq:F01})
\begin{equation}\label{eq:H04}
H_0(u)=\frac{u}{(1-\rho\myp u)^r}=\sum_{k=1}^\infty
\binom{r+k-2}{k-1}\myp\rho^{k-1} u^{k},
\end{equation}
which has positive coefficients $a_k$ (cf.~\eqref{eq:b-k-r}). Hence,
Assumption \ref{as:7.1} is satisfied by virtue of
Remark~\ref{rm:a1>0}. To check the condition $A^+(\sigma)<\infty$,
observe using Stirling's asymptotic formula for the gamma function
(see \cite[\S12.5, p.\,130]{Cramer}) that
\begin{equation*}
a_k=\binom{r+k-2}{k-1}\myp\rho^{k-1}
=\frac{\Gamma(k+r-1)}{\Gamma(r)\mypp\Gamma(k)}\mypp\rho^{k-1}\sim
\frac{k^{\myp r-1}}{\Gamma(r)}\,\rho^{k-1}\qquad (k\to\infty),
\end{equation*}
hence $A^{+}(\sigma)<\infty$ for any $\sigma>0$ if $\rho<1$, whereas
if $\rho=1$ then $A^{+}(\sigma)<\infty$ only for $\sigma>r$.

On substituting \eqref{eq:H04} into Taylor's expansion of the
exponential function in \eqref{eq:beta4a}, it is evident that the
corresponding coefficients $c_{k}$ in the power series \eqref{eq:F0}
of $\mathcal{F}_0(u)$ are also positive, with $c_0=c_1=1$,
$c_2=r\myn\rho+\frac12$, etc.; more generally, $c_k$'s can be
evaluated using Fa\`a di Bruno's formula like in Example \ref{ex:3},
but we omit the details.

However, the special case $\rho=1$, $r=m\in\NN$ may be given a
combinatorial interpretation as follows. Substituting the expansion
\eqref{eq:H04} into \eqref{eq:beta4a} and setting
$m_\ell:=\ell!\mypp \binom{m+\ell-2}{\ell-1}$ for $\ell\in\NN$
(cf.~\eqref{eq:assemblyEx3}), the identity \eqref{eq:exp(M)1}
applied to $\mathcal{F}_0(u)$ yields that the coefficients $c_k$ in
the power series expansion of \eqref{eq:beta4a} are expressed as
$c_k=p(k)/k!$\myp, where $p(k)$ is the total number of instances of
the corresponding assembly of size $k$. Construction of such an
assembly comprises three steps: (i) the set $\{1,\dots,k\}$ is
partitioned into non-empty blocks; (ii) a block with $\ell$ elements
is represented as a linear rooted tree (see Remark~\ref{rm:Faa})
distinguished by $\ell!$ permutations of its vertices; (iii)
\myp$\ell-1$ edges of such a tree are colored using $m$ colors,
subject to the convention that if $j$-th color is used $i_j\ge0$
times (with $i_1+\dots+i_{m}=\ell-1$) then the color schemes are
distinguishable only if the corresponding \strut{}$m$-tuples
$(i_1,\dots,i_{m})$ are not identical, making the total number of
such schemes equal to $\binom{m+\ell-2}{\ell-1}$ (see \cite[Ch.\,II,
\S\myp5, p.\,38]{Feller1}).
\end{example}

\begin{example}\label{ex:5}
Combining the exponential form of Example \ref{ex:4} with the
generating function from Example \ref{ex:2}, for $\rho\in[0,1]$,
\,$m\in\NN$ consider
\begin{equation}\label{eq:F05}
\mathcal{F}_0(u):=\exp\mynn\bigl(u\myp(1+\rho\myp
u)^{m-1}\bigr),\qquad u\in\CC.
\end{equation}
Since $u\mapsto u\myp(1+\rho\myp u)^{m-1}$ is a polynomial of degree
$m$ with positive coefficients, it follows that the coefficients
$c_k$ in the power series expansion of the function \eqref{eq:F05}
are positive for all $k\in\ZZ_+$\myp.
\begin{remark}\label{rm:real_r}
Caution is needed with a \emph{non-integer} $r>1$ replacing
$m\in\NN$ in \eqref{eq:F05}: e.g., for $\rho=1$, \mypp$r=1.5$ we
obtain (with the help of
\texttt{Maple\myp\footnotesize\texttrademark})
$c_9=-921479/92897280<0$.
\end{remark}

From \eqref{eq:F05} by the binomial formula we obtain
\begin{equation}\label{eq:H05}
H_0(u)=u\myp(1+\rho\myp
u)^{m-1}=\sum_{k=1}^{m}\binom{m-1}{k-1}\myp\rho^{k-1} u^{k},
\end{equation}
so that the corresponding coefficients $a_k$'s are positive for
$k=1,\dots,m$ and vanish for $k\ge m+1$. Hence, Assumption
\ref{as:7.1} is satisfied and $A^+(\sigma)<\infty$ for any
$\sigma>0$.

In the special case $\rho=1$, it is not hard to give a combinatorial
interpretation of the coefficients $c_k$ by adapting considerations
in Examples \ref{ex:3} and~\ref{ex:4}. Indeed, substituting the
expansion \eqref{eq:H05} back into \eqref{eq:F05} and defining
$m_\ell:=\ell!\mypp\binom{m-1}{\ell-1}$ for $\ell=1,\dots,m$ and
$m_\ell\equiv 0$ for $\ell\ge m+1$, similarly as above we can use
the exponential identity \eqref{eq:exp(M)1} to conclude that
$c_k=p(k)/k!$\myp, where $p(k)$ is the total number of assemblies of
size $k$ constructed as follows: (i) the set $\{1,\dots,k\}$ is
partitioned into blocks of size \emph{not bigger than $m$} each;
(ii) a block of size $\ell$ is arranged as a rooted linear tree with
labeled vertices (resulting in $\ell!$ possible permutations); (iii)
the total of $m$ unlabeled tokens is allocated to $\ell$ consecutive
vertices on such a tree according to an integer $\ell$-tuple
$(m_1,\dots,m_\ell)$ subject to the conditions
$m_1+\dots+m_{\ell}=m$ and $m_i\ge1$ for all $i=1,\dots,\ell$ (so
that all $m$ tokens are allocated and each vertex gets at least one
token); the total number of such (strict) allocations is known to be
given by $\binom{m-1}{\ell-1}$ (see \cite[Ch.\,II, \S\myp5,
p.\,38]{Feller1}).
\end{example}

\begin{example}\label{ex:6}
For $\rho\in(0,1]$, \myp$r\in(0,\infty)$, consider the generating
function
\begin{equation}\label{eq:F06}
\mathcal{F}_0(u):=\left(\frac{-\ln\myn(1-\rho\myp u)}{\rho\myp
u}\right)^r\!\equiv (f_0(u))^r,\qquad |u|<\rho^{-1},
\end{equation}
where
\begin{equation}\label{eq:f0}
f_0(u):=\frac{-\ln\myn(1-\rho\myp u)}{\rho\myp
u}=1+\sum_{k=1}^\infty \frac{\rho^k u^{k}}{k+1}\myp.
\end{equation}
If $r=m\in\NN$ then from \eqref{eq:f0} it is evident that the
coefficients $c_k$ in the power series expansion of
$\mathcal{F}_0(u)$ in \eqref{eq:F06} are positive for all
$k\in\ZZ_+$\myp; however, for non-integer $r>0$ this is not so
clear, since the binomial expansion of $t\mapsto (1+t)^r$ involves
negative terms (cf.\ Remark~\ref{rm:real_r}). Yet, as a matter of
fact, the positivity of $c_k$'s holds for \emph{any real} $r>0$ ---
this will be established in Corollary \ref{cor:0<c<}.

By a term-by-term comparison, it is also clear that, for any $r>0$,
\begin{equation}\label{eq:F<F}
\mathcal{F}_0(u)=\left(1+\sum_{k=1}^\infty \frac{\rho^k
u^{k}}{k+1}\right)^{\myn r}\mynn\le \left(1+\sum_{k=1}^\infty \rho^k
u^{k}\right)^{\myn r}\mynn=(1-\rho\myp u)^{-r},\qquad 0\le
u<\rho^{-1},
\end{equation}
with the inequality being strict for $u>0$. That is to say, the
function $\mathcal{F}_0(u)$ is bounded by a multiset-type generating
function \eqref{eq:F01} considered in Example \ref{ex:1}. Moreover,
\emph{for integer} $r=m\in\NN$, by expanding both parts in
\eqref{eq:F<F} it is evident that the coefficients $c_k$ in the
power series expansion of $\mathcal{F}_0(u)$ are dominated by the
coefficients of the multiset generating function $(1-\rho\myp
u)^{-m}$  (cf.~\eqref{eq:b-k-r}),
\begin{equation}\label{eq:b-k-r-Ex6}
c_k<\binom{m+k-1}{k}\myp\rho^k=\frac{m\myp(m+1)\cdots(m+k-1)}{k!}\,
\rho^k,\qquad k\in\NN.
\end{equation}
Thus, the multiplicative ensemble determined by \eqref{eq:F06} may
be viewed (at least for integer $r=m$) as a \emph{discounted
multiset ensemble}, whereby larger values of each count $\nu_\ell=k$
are progressively discouraged. Again, this statement turns out to be
true for \emph{any real\/ $r>0$}, which will be explained below (see
Corollary \ref{cor:0<c<}). On the other hand, a direct combinatorial
interpretation of the generating function \eqref{eq:F06} (say, in
the spirit of the previous examples) is not clear, even in the
simplest case $r=\rho=1$.

Let us now look at the function $H_0(u)=\ln\myn(\mathcal{F}_0(u))$
(see~\eqref{eq:H}): according to \eqref{eq:F06},
\begin{equation}\label{eq:H06}
H_0(u)=r\ln\!\myn\left( \frac{-\ln\myn(1-\rho\myp u)}{\rho\myp u}
\right) =r\ln\myn(f_0(u)),\qquad |u|<\rho^{-1}.
\end{equation}
The next proposition implies that $ A^+(\sigma)<\infty$ for any
$\sigma>0$ (including the case $\rho=1$); furthermore, since all
$a_k>0$, by Remark~\ref{rm:a1>0} it follows that Assumption
\ref{as:7.1} is satisfied.
\end{example}

\begin{proposition}\label{pr:a-ex5}
The coefficients $\{a_k\}$ in the power series expansion
\eqref{eq:H} of the function \eqref{eq:H06} satisfy the inequalities
\begin{equation}\label{eq:<a<}
\frac{r\rho^k}{k^2\myp(k+1)}\le a_k\le\frac{r\rho^k}{k+1}\myp,\qquad
k\in\NN.
\end{equation}
In particular, $a_k>0$ for all $k\in\NN$.
\end{proposition}

\begin{proof}
Differentiation of the identity $r\ln\myn(f_0(u))=\sum_{j=1}^\infty
a_j\myp u^j$ (see \eqref{eq:H06}) gives
\begin{equation}\label{eq:dif-once}
r f_0^{\prime}(u)=f_0(u)\sum_{j=1}^\infty ja_j\myp u^{j-1}.
\end{equation}
Differentiating \eqref{eq:dif-once} again $k-1$ times ($k\ge1$), by
the Leibniz rule we obtain
\begin{equation}\label{eq:f'(0)}
f_0^{(k)}(0)=\frac{1}{r}\sum_{i=0}^{k-1} \binom{k-1}{i}\myp
f_0^{(k-1-i)}(0)\, (i+1)! \,a_{i+1},\qquad k\in\NN.
\end{equation}
Noting from \eqref{eq:f0} that $f_0^{(j)}(0)=\rho^j j!\myp/(j+1)$
($j\in\ZZ_+$) and using the shorthand notation $\tilde{a}_k:=k
a_k\rho^{-k}\myn /r$ \myp($k\in\NN$), the system of equations
\eqref{eq:f'(0)} is reduced to
\begin{equation}\label{eq:a_m+1}
\frac{k}{k+1}=\sum_{i=0}^{k-1}
\frac{\tilde{a}_{i+1}}{k-i}\myp,\qquad k\in\NN,
\end{equation}
while the inequalities \eqref{eq:<a<} are rewritten as
\begin{equation}\label{eq:<a<*}
\frac{1}{k\myp(k+1)}\le \tilde{a}_k\le \frac{k}{k+1}\myp,\qquad
k\in\NN.
\end{equation}

Let us prove \eqref{eq:<a<*} by induction in $k\in\NN$. For $k=1$,
from \eqref{eq:a_m+1} we find $\tilde{a}_1=\frac12$ and the claim
\eqref{eq:<a<*} is obviously satisfied. Suppose now that the
inequalities \eqref{eq:<a<*} hold for $\tilde{a}_1,\dots,
\tilde{a}_{k-1}$ ($k\ge2$), which entails that $\tilde{a}_i>0$
($i=1,\dots,k-1$). Then the recursion \eqref{eq:a_m+1} (with $k$
replaced by $k-1$) implies
\begin{align*}
\frac{k}{k+1} &=\sum_{i=0}^{k-2}
\frac{\tilde{a}_{i+1}}{k-i}+\tilde{a}_{k}\le \sum_{i=0}^{k-2}
\frac{\tilde{a}_{i+1}}{k-1-i}+\tilde{a}_{k}=\frac{k-1}{k}+\tilde{a}_{k},
\end{align*}
and it follows that
\begin{equation*}
\tilde{a}_{k}\ge
\frac{k}{k+1}-\frac{k-1}{k}=\frac{1}{k\myp(k+1)}\myp,
\end{equation*}
which gives the lower bound in \eqref{eq:<a<*}. On the other hand,
again using that $\tilde{a}_1,\dots, \tilde{a}_{k-1}>0$, from
\eqref{eq:a_m+1} we get
\begin{equation*}
\frac{k}{k+1}=\tilde{a}_{k}+\sum_{i=0}^{k-2}
\frac{\tilde{a}_{i+1}}{k-i}\ge \tilde{a}_{k},
\end{equation*}
which proves the upper bound in \eqref{eq:<a<*}. Thus, the claim
\eqref{eq:<a<*} is verified for the $\tilde{a}_k$, and therefore it
is valid with all $k\in\NN$.
\end{proof}

\begin{corollary}\label{cor:0<c<}
For any real $r>0$, the coefficients $c_k$ in the power series
expansion of the generating function \eqref{eq:F06} satisfy the
two-sided bounds \textup{(}cf.~\eqref{eq:b-k-r-Ex6}\textup{)}
\begin{equation}\label{eq:b-k-r-Ex6-r}
0<c_k<\binom{r+k-1}{k}\myp\rho^k, \qquad k\in\NN.
\end{equation}
\end{corollary}

\begin{proof}
Using the expansion $H_0(u)=\sum_{k=1}^\infty\myn a_k u^k$ we have
\begin{equation}\label{eq:F-exp(H)}
\mathcal{F}_0(u)=\exp\myn(H_0(u))=\exp\!\left(\sum_{k=1}^\infty a_k
u^k\right)=1+\sum_{k=1}^\infty c_k u^k.
\end{equation}
By Proposition \ref{pr:a-ex5} all $a_k>0$, and since Taylor's
coefficients of the exponential function are positive as well, it is
evident from \eqref{eq:F-exp(H)} that $c_k>0$ for all $k\in\NN$.

Furthermore, from the bounds \eqref{eq:<a<} we get
\begin{equation}\label{eq:a<1/k}
0<a_k\le \frac{r\rho^k}{k+1}<\frac{r\rho^k}{k}\myp,\qquad k\in\NN.
\end{equation}
Considering the corresponding power series and their exponentials
\begin{equation*}
\exp\!\left(\myn\sum_{k=1}^\infty a_k u^k\right)=
\exp\mynn\bigl(H_0(u)\bigr)\equiv\mathcal{F}_0(u)
\end{equation*}
and
\begin{equation*}
\exp\!\left(\myn\sum_{k=1}^\infty \frac{\,r\rho^k }{k}\,u^k
\right)=\exp\myn\bigl(-r\ln\myn(1-\rho\myp u) \bigr)=(1-\rho\myp
u)^{-r}=:\widetilde{\mathcal{F}}(u),
\end{equation*}
it follows from the term-by-term subordination \eqref{eq:a<1/k} that
the coefficients in the respective power series expansions
$\mathcal{F}_0(u)=\sum_k c_ku^k$ and
$\widetilde{\mathcal{F}}(u)=\sum_k \tilde{c}_ku^k$ inherit the same
(strict) subordination, that is, $c_k<\tilde{c}_k$ for all
\strut{}$k\in\NN$. It remains to notice that (cf.\ \eqref{eq:F01},
\eqref{eq:b-k-r}) $\tilde{c}_k=\binom{r+k-1}{k}\myp\rho^k$, which
yields the upper bound in~\eqref{eq:b-k-r-Ex6-r}, as claimed.
\end{proof}

For convenience, the results of Section \ref{sec6.1.2} are
summarized in Table~\ref{t1}.

\begin{table}[ht]
\begin{center}
\centerline{\parbox{.92\textwidth}{\caption{The generating functions
in Examples \ref{ex:1}--\mypp\ref{ex:6} \,(\myp$0<\rho\le 1$,
\,$r>0$, \,$m\in\NN$, \,$b>0$\myp). Third column shows the type of
singularity of $\mathcal{F}_0(u)$ (with $\rho=1$, \,$r=m\in\NN$) at
point $u=1$.}\label{t1}}}  \vspace{.4pc} \footnotesize
\tabcolsep=0.22pc
\begin{tabular}{|c|c|c|c|c|c|}
\hline &&&&&\\[-.7pc]
\,No.&$\displaystyle\mathcal{F}_0(u)$
    &
    $u=1$&$\displaystyle H_0(u)$
     &$a_k \;(k\in\NN)$
      &$\displaystyle A^+(\sigma)<\infty$
       \\[.2pc]
\hline &&&&&\\[-.9pc]
\hline &&&&&\\[-.8pc]
\ref{ex:1}
   &$\displaystyle(1-\rho\myp u)^{-r}$
    &$\begin{array}{c}
    \text{pole of}\\[-.1pc]
    \text{order $m$}
    \end{array}$&$\displaystyle-r\ln\mynn(1-\rho\myp u)$
     &$r \myp k^{-1}\myn\rho^k$
      &$\sigma>0$
        \\[.5pc]
\hline &&&&&\\[-.8pc]
\ref{ex:2}
   &$\displaystyle (1+\rho\myp u)^{m}$
    &$\begin{array}{c}
        \text{regular}\\[-.1pc]
        \text{point}
     \end{array}$&$\displaystyle m\ln\mynn(1+\rho\myp u)$
     &$(-1)^{k-1}\myp r \myp k^{-1}\myn\rho^k$
      &$\displaystyle\sigma>0$
       \\[.3pc]
\hline &&&&&\\[-.7pc]
\ref{ex:3}
   &$\displaystyle\exp\mynn\biggl(\mbox{\footnotesize$\displaystyle
                                    \frac{b\myp u}{1- \rho\myp u}$}
                           \biggr)$
    &$\begin{array}{c}
         \text{essential}\\[-.1pc]
          \text{singularity}
      \end{array}$
    &$\displaystyle \frac{b\myp u}{1-\rho\myp u}$
     &$\displaystyle b\mypp \rho^{k-1}$
      &$\begin{array}{l}
               \sigma>0 \ \ (\rho<1)\\[.1pc]
               \sigma>1 \ \ (\rho=1)
        \end{array}$
        \\[.7pc]
\hline &&&&&\\[-.7pc]
\ref{ex:4}
   &$\ \displaystyle\exp\myn\biggl(\frac{u}{(1-\rho\myp u)^{r}}\biggr)$
    &$\begin{array}{c}
    \text{essential}\\[-.1pc]
      \text{singularity}
      \end{array}$
      &$\displaystyle \frac{u}{(1-\rho\myp u)^{r}}$
       &$\begin{array}{c}
       \mbox{\scriptsize$\left(\!\!\myn\begin{array}{c}
       r+k-2\\k-1\end{array}\!\!\right)$}\myp\rho^{k-1}\\[.6pc]
       \sim k^{\myp r-1}\myn\rho^{k-1}\mynn/\Gamma(r)
       \end{array}$
      &$\begin{array}{l}
              \sigma>0 \ \ (\rho<1)\\[.1pc]
              \sigma>r \ \ (\rho=1)
                 \end{array}$
       \\[1.1pc]
\hline &&&&&\\[-.7pc]
\ref{ex:5}
    &\ $\displaystyle\exp\mynn\bigl(u\myp(1+\rho\myp u)^{m-1}\bigr)$
    &$\begin{array}{c}
    \text{regular}\\[-.1pc]
      \text{point}
      \end{array}$
      &$\displaystyle u\myp(1+ \rho\myp u)^{m-1}$
       &$\begin{array}{c}
         \mbox{\scriptsize $
         \left(\!\!\mynn\begin{array}{c}m-1\\k-1
                     \end{array}\!\!\mynn\right)$}\myp\rho^{k-1}\\[.4pc]
         (k=1,\dots,m)
       \end{array}$
      &$\sigma>0$
        \\[1.0pc]
\hline &&&&&\\[-.7pc]
\ref{ex:6}
   &$\displaystyle\left(\frac{-\ln\mynn(1-\rho\myp u)}{\rho\myp u}\right)^r$
    &$\begin{array}{c}
    \text{branch}\\[-.15pc]
      \text{point}
      \end{array}$
    &$\displaystyle \ r\ln
    \!\mynn\left(\frac{-\ln\mynn(1-\rho\myp u)}{\rho\myp u}\right)
    $
     &$\displaystyle a_k=O(k^{-1}\mynn\rho^k)$
      &$\sigma>0$
       \\[.8pc]
\hline
\end{tabular}
\end{center}
\end{table}

\subsection{The limit shapes}\label{sec6.3}
In this section, for each of Examples \ref{ex:1}--\mypp\ref{ex:6} we
evaluate the parameter $\gamma=\sqrt{A(1)}$ (see \eqref{eq:kappa})
using for $A(1)$ either the definition \eqref{eq:A2-} or the
equivalent integral expression \eqref{eq:kappa1}, and then apply
Theorems \ref{th:LSQ} and \ref{th:LSP} to obtain the explicit limit
shape $\omega^*(x)$ as identified by the general formula
\eqref{eq:omega}. For the reader's convenience, the limit shape
function in Example~6.$i$ is denoted by $\omega^*_{i}(x)$
($i=1,\dots,6$), and the results are summarized in Table~\ref{t2}.
\begin{table}[ht]
\begin{center}
\caption{The limit shapes in Examples \ref{ex:1}--\mypp\ref{ex:6}
\,(\myp$0<\rho\le 1$, \,$r>0$, \,$m\in\NN$, \,$b>0$\myp).}
\label{t2} \vspace{.4pc} \footnotesize
\tabcolsep=.4pc
\begin{tabular}[h]{|c|c|c|c|c|}
\hline &&&&\\[-.6pc]
\,No.
  &$\displaystyle H_0(u)$
   &$\gamma^2$
    &$\displaystyle \omega_i^*(x)$
     &Special case
\\[.2pc]
\hline &&&&\\[-.9pc]
\hline &&&&\\[-.7pc]
\ref{ex:1}
   &$\displaystyle-r\ln\mynn(1-\rho\myp u)$
    &$\displaystyle r\Li_2(\rho)$
     &$\displaystyle -\frac{r\myp\ln\mynn(1-\rho\mypp\rme^{-\gamma x})}{\gamma}$
      &$\begin{array}{c}
          r=\rho=1\mynn\!:\,\\
          \mynn\gamma=\pi/\sqrt{6}
        \end{array}$
\\[.7pc]
\hline &&&&\\[-.7pc]
\ref{ex:2}
   &$\displaystyle m\myp\ln\mynn(1+\rho\myp u)$
    &$\ \displaystyle
     -m\Li_2(-\rho)$
     &$\displaystyle\frac{m\myp\ln\mynn(1+\rho\mypp\rme^{-\gamma x})}{\gamma}$
      &\!$\begin{array}{c}
                         \rho=1, \,m=1\mynn\!:\\
                         \gamma=\pi/\sqrt{12}
                      \end{array}$
      \\[.7pc]
\hline &&&&\\[-.7pc]
\ref{ex:3}
   &$\displaystyle \frac{b\myp u}{1- \rho\myp u}$
    &$\displaystyle-\frac{b\myp\ln\mynn(1-\rho)}{\rho}$
     &$\displaystyle\frac{b\,\rme^{-\gamma x}}{\gamma\mypp(1-\rho\mypp\rme^{-\gamma x})}$
      &$\begin{array}{c}
                 \rho\to0^+{:}\\[.1pc]
                 \gamma=\sqrt{b}
               \end{array}$
       \\[.8pc]
\hline &&&&\\[-.7pc]
\ref{ex:4}
   &$\displaystyle \frac{u}{(1-\rho\myp u)^r}$
    &$\displaystyle
    \frac{1-(1-\rho)^{1-r}}{\rho\myp(1-r)}$
     &$\displaystyle\frac{\rme^{-\gamma x}}
                         {\gamma\mypp(1-\rho\mypp\rme^{-\gamma
                         x})^r}$
      &$\begin{array}{c}
                     \rho=1,\ r<1\myn\!:\\[.1pc]
                     \gamma=1/\sqrt{1-r}
                    \end{array}$
        \\[.8pc]
\hline &&&&\\[-.7pc]
\ref{ex:5}
   &$\displaystyle u\myp(1+ \rho\myp u)^{m-1}$
    &$\displaystyle\frac{(1+\rho)^{m}-1}{\rho\mypp m}$
     &$\displaystyle\frac{\rme^{-\gamma x}\mypp(1+\rho\mypp
                      \rme^{-\gamma x})^{m-1}}{\gamma}$
      &$\begin{array}{c}
                      \rho=1, \,m=2\mynn\!:\\[.1pc]
                      \gamma=\sqrt{3/2}
                     \end{array}$
       \\[.8pc]
\hline &&&&\\[-.7pc]
\ref{ex:6}
   &\ $\displaystyle r\ln\!\myn\mynn\left(\frac{-\ln\mynn(1-\rho\myp u)}{\rho\myp u}\right)$
    &$\int_0^1\myn u^{-1} H_0(u)\,\dif{u}$
     &$\ \displaystyle\frac{r}{\gamma}\myp\ln\!\myn\mynn
        \left(\frac{-\ln\mynn(1-\rho\mypp\rme^{-\gamma x})}{\rho\mypp\rme^{-\gamma x}}
        \right)$
      &$\begin{array}{c}
                   \mypp r=\rho=1\mynn\!:\\[.1pc]
                   \gamma\doteq0.853636
                  \end{array}$
       \\[.8pc]
\hline
\end{tabular}
\end{center}
\end{table}

\smallskip
Starting with Example \ref{ex:1}, from \eqref{eq:kappa} and
\eqref{eq:H01} we have
\begin{equation}\label{eq:gamma1}
\gamma^2=r\sum_{k=1}^\infty \frac{\rho^k}{k^2}=r\Li_2(\rho),
\end{equation}
where $\Li_2(\cdot)$ is the \emph{dilogarithm} (see \cite{Lewin}).
Hence, the limit shape is given by (see~\eqref{eq:omega})
\begin{equation}\label{eq:LS1}
\fbox{\ $\displaystyle
\omega^*_1(x)=-\frac{r\vphantom{r^r}}{\gamma}\,\ln\myn(1-\rho\,\rme^{-\gamma
x})$}
\end{equation}
For $r=\rho=1$, \eqref{eq:gamma1} gives $\gamma^2=\Li_2(1)=\pi^2/6$,
and from \eqref{eq:LS1} we recover the classical formula
\eqref{eq:limit1} for the limit shape of partitions under the
uniform distribution on~$\CP_n$\myp.

\smallskip
Next, in Example \ref{ex:2} we have, according to
\eqref{eq:kappa} and~\eqref{eq:H02},
\begin{equation}\label{eq:gamma2}
\gamma^2=m\sum_{k=1}^\infty
\frac{(-1)^{k-1}\rho^k}{k^2}=-m\Li_2(-\rho)\equiv m\myp\bigl(
\Li_2(\rho)-\tfrac12 \Li_2(\rho^2)\bigr),
\end{equation}
and the limit shape \eqref{eq:omega} specializes to
\begin{equation}\label{eq:LS2}
\fbox{\ $\displaystyle
\omega^*_2(x)=\frac{m\vphantom{r^r}}{\gamma}\,\ln\myn(1+\rho\,\rme^{-\gamma
x})$}
\end{equation}
If $m=\rho=1$ then from \eqref{eq:gamma2} we find
$\gamma^2=\frac12\Li_2(1)=\pi^2/12$ and the equation \eqref{eq:LS2}
is reduced to the limit shape \eqref{eq:limit2} of uniformly
distributed strict partitions.

\smallskip
In Example \ref{ex:3}, according to \eqref{eq:kappa}
and~\eqref{eq:H03} we have for $\rho\in(0,1)$
\begin{equation}\label{eq:gamma3}
\gamma^2=b\sum_{k=1}^\infty
\frac{\rho^{k-1}}{k}=-\frac{b}{\rho}\myp\ln\myn(1-\rho)<\infty,
\end{equation}
while if $\rho=0$ then $\gamma^2=b$. Alternatively, we can obtain
the same result as \eqref{eq:gamma3} by using the integral formula
\eqref{eq:kappa1} with the expression \eqref{eq:H03} for $H_0(u)$,
\begin{equation}\label{eq:gamma3a}
\gamma^2=\int_0^1\mynn \frac{b}{1-\rho\myp
u}\,\dif{u}=-\frac{b}{\rho}\myp\ln\myn(1-\rho).
\end{equation}
In turn, equation \eqref{eq:omega} for the limit shape is reduced to
\begin{equation*}
\fbox{\ $\displaystyle \omega^*_3(x)=\frac{b\mypp\rme^{-\gamma
x}}{\gamma\myp(1-\rho\mypp\rme^{-\gamma x})}$}
\end{equation*}

\medskip Likewise, in Example \ref{ex:4} we get, similarly to
\eqref{eq:gamma3a},
\begin{equation}\label{eq:gamma4a}
\gamma^2=\int_0^1\mynn\frac{1}{(1-\rho\myp
u)^r}\,\dif{u}=\frac{1-(1-\rho)^{1-r}}{\rho\myp(1-r)}<\infty,
\end{equation}
which holds for $0<\rho<1$ (and any $r\ne1$). In the special case
$\rho=1$ the computation in \eqref{eq:gamma4a} is modified as
follows,
\begin{equation*}
\gamma^2=\int_0^1\mynn
\frac{1}{(1-s)^r}\,\dif{s}=\frac{1}{1-r}<\infty,
\end{equation*}
provided that $r<1$. Next, substituting \eqref{eq:H04} into
\eqref{eq:omega}, we get the limit shape
\begin{equation*}
\fbox{\ $\displaystyle \omega^*_4(x)=\frac{\rme^{-\gamma
x}}{\gamma\mypp(1-\rho\mypp\rme^{-\gamma x})^r}$}
\end{equation*}

\medskip In Example \ref{ex:5} we have, according to
\eqref{eq:kappa1} and~\eqref{eq:H05},
\begin{equation*}
\gamma^2=\int_0^1\mynn (1+\rho\myp
u)^{m-1}\,\dif{u}=\frac{(1+\rho)^{m}-1}{\rho\mypp m}\myp,
\end{equation*}
and the limit shape \eqref{eq:omega} specializes to
\begin{equation*}
\fbox{\ $\displaystyle \omega^*_5(x)=\frac{\rme^{-\gamma
x}\mypp(1+\rho\mypp\rme^{-\gamma x})^{m-1}}{\gamma}$}
\end{equation*}
In particular, if $\rho=1$, $m=2$, then $\gamma^2=\frac32$.

\smallskip Finally, in Example \ref{ex:6} the parameter $\gamma$ may
only be computed numerically, to which end it is more convenient to
use the formula \eqref{eq:kappa1}. The limit shape can then be
plotted using the explicit equation \eqref{eq:omega} with the
function $H_0(\rme^{-\gamma x})$ evaluated from the formula
\eqref{eq:H06},
\begin{equation*}
\fbox{\ \myp$\displaystyle
\omega^*_6(x)=\frac{r\vphantom{X^x}}{\gamma\vphantom{y_y}}\mypp\Bigl\{\gamma
x-\ln\myn\rho+\ln\mynn\bigl( -\ln\myn(1-\rho\mypp \rme^{-\gamma
x})\bigr)\Bigr\}$}
\end{equation*}
For instance, taking $r=1$, $\rho=0.5$ we computed\mypp\footnote{\
Numerical computations and graphical outputs were obtained using
\texttt{Maple\myp\scriptsize\texttrademark}.}
$\gamma\doteq0.532202$,
with the corresponding limit shape shown in Fig.\,\ref{fig3}a. For a
comparison, we also plotted the limit shape with parameters $r=1$,
$\rho=1$, giving $\gamma\doteq0.853636$
(see Fig.\,\ref{fig3}b).

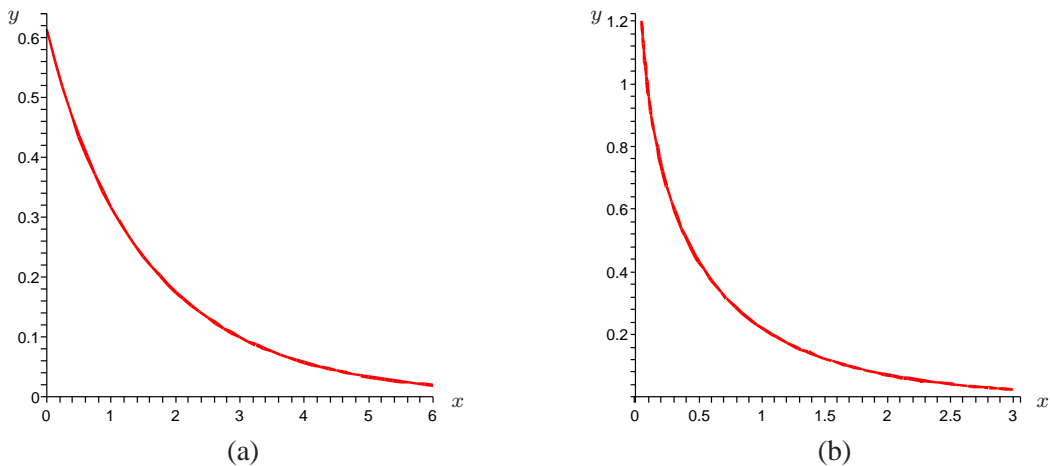
\begin{figure}[h]
\thinlines
\mbox{}\hspace{2.5pc}\includegraphics[width=2.50in]{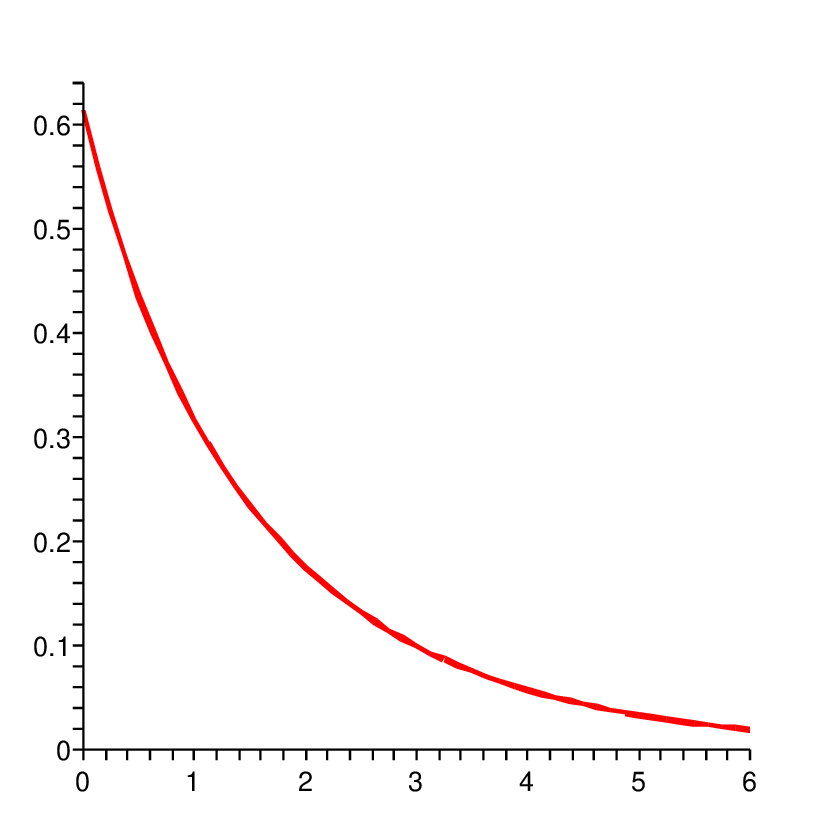}\hspace{.5in}
\includegraphics[width=2.50in]{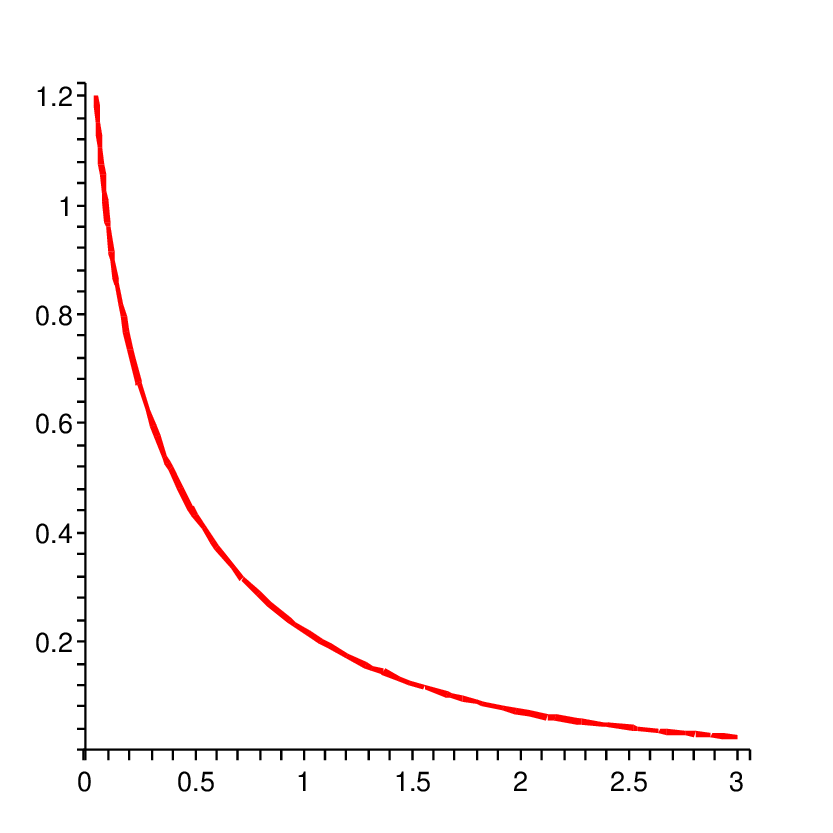}
\put(-315,-6){\mbox{\small (a)}} %
\put(-94,-6){\mbox{\small (b)}}  %
\put(-231,14){\mbox{\scriptsize$x$}}
\put(-12,14){\mbox{\scriptsize$x$}}
\put(-397,160){\mbox{\scriptsize$y$}}
\put(-179,160){\mbox{\scriptsize$y$}} %
\caption{The limit shape $y=\omega^*(x)$ in Example \ref{ex:6}, with
the function $H_0(u)$ given by equation \eqref{eq:H06}: \,(a)
\,$r=1$, \myp$\rho=0.5$ \,($\gamma\doteq0.532202$); \;(b)
\,$r=\rho=1$ \,($\gamma\doteq0.853636$).} \label{fig3}
\end{figure}

\section*{Acknowledgments}\label{ack}
This work was supported in part by a Leverhulme Research Fellowship.
Partial support by the Hausdorff Research Institute for Mathematics
(Bonn) is also acknowledged. The author is grateful to Boris
Granovsky, Anatoly Vershik and Yuri Yakubovich for helpful
discussions, and to the anonymous referees for constructive comments
that helped to improve the presentation.

\end{document}